\documentclass{article} 
\newif\ifarxiv
\arxivtrue 
\ifarxiv
  \usepackage{iclr2027_conference_hack}
\else
  \usepackage{iclr2027_conference}
\fi
\usepackage{times}

\usepackage[utf8]{inputenc} 
\usepackage[T1]{fontenc}    
\usepackage[hypertexnames=false]{hyperref}       
\usepackage{url}            
\usepackage{booktabs}       
\usepackage{amsfonts}       
\usepackage{nicefrac}       
\usepackage{microtype}      
\usepackage{xcolor}         

\usepackage{adjustbox}
\usepackage{subcaption}

\newcommand{\markchanges}{}

\newcommand{\squeeze}{\textstyle}

\makeatletter
\newenvironment{protocol}[1][htb]{%
    \renewcommand{\ALG@name}{Protocol}
   \begin{algorithm}[#1]%
  }{\end{algorithm}}
\makeatother



\usepackage{tcolorbox}

\newtcolorbox{theorembox}{
  colback=gray!20,
  colframe=gray!20,
  boxrule=0.8pt,
  before skip=5pt,
  after skip=5pt,
  boxsep=-1mm,
}

\usepackage{graphicx}
\usepackage{apptools}
\usepackage[flushleft]{threeparttable}
\usepackage{array,booktabs,makecell}
\usepackage{multirow}
\usepackage{nicefrac}
\usepackage{amsthm}
\usepackage{amsmath,amsfonts,bm}
\usepackage{wrapfig}
\usepackage{caption}
\usepackage{siunitx}
\usepackage{thm-restate}
\usepackage{nccmath}
\usepackage{empheq}
\usepackage{bbm}
\usepackage{tabularx}

\usepackage{algorithmic}
\usepackage{algorithm}
\usepackage{filecontents}

\usepackage{color}
\usepackage{colortbl}
\definecolor{bgcolor}{rgb}{0.76,0.88,0.50}
\definecolor{bgcolor0}{rgb}{0.93,0.99,1}
\definecolor{bgcolor1}{rgb}{0.8,1,1}
\definecolor{bgcolor2}{rgb}{0.8,1,0.8}
\definecolor{bgcolor3}{rgb}{0.50,0.90,0.50}
\usepackage{tcolorbox}
\usepackage{pifont}
\definecolor{mydarkgreen}{rgb}{39,130,67}
\definecolor{mydarkred}{rgb}{192,25,25}

\newcommand{\norm}[1]{\left\| #1 \right\|}

\newcommand{\inp}[2]{\left\langle#1,#2\right\rangle} 
\newcommand{\abs}[1]{\left| #1 \right|}
\newcommand{\flr}[1]{\left\lfloor #1\right\rfloor} 
\newcommand{\ceil}[1]{\left\lceil #1\right\rceil} 

\newcommand{\R}{\mathbb{R}} 
\newcommand{\N}{\mathbb{N}} 

\newcommand{\Exp}[1]{{\mathbb{E}}\left[#1\right]}
\newcommand{\ExpSub}[2]{{\mathbb{E}}_{#1}\left[#2\right]}

\newcommand{\cA}{\mathcal{A}}

\newcommand{\cF}{\mathcal{F}}

\newcommand{\cO}{\mathcal{O}}

\newcommand{\mA}{\mathbf{A}}

\newcommand{\mI}{\mathbf{I}}

\theoremstyle{plain}
\newtheorem{theorem}{Theorem}[section]

\newtheorem{lemma}[theorem]{Lemma}

\theoremstyle{definition}
\newtheorem{definition}[theorem]{Definition}
\newtheorem{assumption}[theorem]{Assumption}
\theoremstyle{remark}
\newtheorem{remark}[theorem]{Remark}

\newcommand{\eqdef}{:=}

\makeatletter
\newcommand{\vast}{\bBigg@{4}}

\usepackage[scaled=0.86]{helvet}
\newcommand{\algname}[1]{{\sf #1}}

\newcounter{takeaway}

\newenvironment{takeawaybox}
  {\refstepcounter{takeaway}\begin{theorembox}\textbf{Takeaway \thetakeaway: }}
  {\end{theorembox}}

\usepackage[textsize=tiny]{todonotes}

\allowdisplaybreaks

\hypersetup{
  colorlinks   = true, 
  urlcolor     = blue, 
  linkcolor    = {red!75!black}, 
  citecolor   = {blue!50!black} 
}

\title{Bridging the Gap Between Homogeneous and Heterogeneous Asynchronous Optimization Is Surprisingly Difficult}

\author{
   Alexander Tyurin \\
   \phantom{,}AXXX, Moscow, Russia \\
   \phantom{,}Applied AI Institute, Moscow, Russia \\
}

\ifarxiv
\iclrfinalcopy 
\fi
\begin{document}

\maketitle

\begin{abstract}
Modern large-scale machine learning tasks often require multiple workers, devices, CPUs, or GPUs to compute stochastic gradients in parallel and asynchronously to train model weights. Theoretical results typically distinguish between two settings: (i) the homogeneous setting, where all workers have access to the \emph{same} data distribution, and (ii) the heterogeneous setting, where each worker operates on \emph{different} data distributions. Known optimal time complexities in these settings reveal a significant gap, with far more pessimistic guarantees in the heterogeneous case. In this work, we investigate whether these pessimistic optimal time complexities can be overcome under different assumptions. Surprisingly, we show that improvement is provably impossible under widely used first- and second-order similarity assumptions for any randomized algorithm. We then turn to the interpolation regime and demonstrate that the weak interpolation assumption alone is also insufficient. Finally, we introduce a minimal combination of irreducible assumptions, strong interpolation and the local Polyak-\L{}ojasiewicz condition, to derive a new time complexity bound that matches the dependence on worker computation times in the best-known result in the homogeneous setting, without requiring identical data distributions.
\end{abstract}

\section{Introduction}
\label{sec:introduction}
We consider optimization problems described by
\begin{align}
\label{eq:main_task}
\squeeze \min \limits_{x \in \R^d} \Big \{f(x) \eqdef \frac{1}{n} \sum\limits_{i=1}^n \ExpSub{\xi_i \sim \mathcal{D}_i}{f_i(x;\xi_i)}\Big \},
\end{align}
where $f_i\,:\,\R^d \times \mathbb{S}_{\xi_i} \rightarrow \R$ and $\xi_i$ is a random variable with distribution $\mathcal{D}_i$ on $\mathbb{S}_{\xi_i}$  for all $i \in [n].$ Let us denote $f_i(x) \eqdef \ExpSub{\xi_i \sim \mathcal{D}_i}{f_i(x;\xi_i)}.$ In our setup, we have $n$ workers/clients/CPUs/GPUs working in parallel and asynchronously, and each worker $i$ has access only to the stochastic gradient $\nabla f_i(x;\xi_i)$ of the function $f_i$ for all $x \in \R^d.$ 
We concentrate on the standard convergence metric and want to find a (possibly random) point $\bar{x}$ such that ${\mathbb{E}}[\norm{\bar{x} - x_*}^2] \leq \varepsilon,$ where $x_*$ is a solution of \eqref{eq:main_task}. 
Such a problem arises in many machine learning (ML), deep learning, federated learning (FL), and data science problems \citep{konevcny2016federated,mcmahan2017communication,goodfellow2016deep}. 
In general, we use the following standard assumptions from convex stochastic optimization, but each result states which assumptions it requires.
\begin{assumption}[Global smoothness]
  \label{ass:global_lipschitz_constant}
  The function $f$ is differentiable and $L$--smooth, i.e., 
$\norm{\nabla f(x) - \nabla f(y)} \leq L \norm{x - y}$ for all $x, y \in \R^d.$
\end{assumption}
\begin{assumption}[Local smoothness]
  \label{ass:lipschitz_constant}
  The functions $f_i$ are differentiable and $L_i$--smooth.
  We also define $L_{\max} \eqdef \max_{i \in [n]} L_i.$ Note that $L \leq L_{\max}$ (This is why we distinguish Assum.~\ref{ass:global_lipschitz_constant} and~\ref{ass:lipschitz_constant}).
\end{assumption}
\begin{assumption}[Convexity]
  \label{ass:convex}
  The functions $f_i$ are convex for all $i \in [n]$. The function $f$ attains a minimum at a (non-unique) point $x_* \in \R^d.$
\end{assumption}
\begin{assumption}[Unbiased and $\sigma^2$-variance-bounded noise]
  \label{ass:stochastic_variance_bounded}
  For all $x \in \R^d,$ stochastic gradients $\nabla f_i(x;\xi)$ are unbiased and $\sigma^2$-variance-bounded, i.e., ${\rm \mathbb{E}}_{\xi_i}[\nabla f_i(x;\xi_i)] = \nabla f_i(x)$ and 
  ${\rm \mathbb{E}}_{\xi_i}[\|\nabla f_i(x;\xi_i) - \nabla f_i(x)\|^2] \leq \sigma^2$ for all $i \in [n],$ where $\sigma^2 \geq 0.$ 
\end{assumption}
We also consider the case where $f$ satisfies the P\L-condition, which is a much weaker assumption than $\mu$--strong convexity \citep{karimi2016linear}:
\begin{assumption}[Global Polyak-\L ojasiewicz condition]
  \label{ass:pl_global}
  There exists $\mu > 0$ such that $\textstyle \norm{\nabla f(x)}^2 \geq 2 \mu \left(f(x) - f^*\right)$ for all $x \in \R^d,$ where $f^*$ is the finite optimal function value of $f.$
\end{assumption}

We focus on the modern setup where many workers work together in a distributed environment, where the workers can have \emph{arbitrarily computation behaviors} due to hardware delays or network connectivity problems. Most previous works typically assume that the workers have the same performance that does not change over time. In contrast, our focus is on the setting where the \emph{computation times are heterogeneous} and non-constant. 

In the literature, the optimization problem \eqref{eq:main_task} in the asynchronous environment is considered in two regimes: i) \emph{heterogeneous setting}, where the functions $f_i$ can be arbitrarily different; in the context of ML and FL, it means the workers have access to different datasets. ii) \emph{homogeneous setting}, where the functions $f_i$ are equal; in the context of ML and FL, it means the workers have access to the same dataset  \citep{koloskova2022sharper,mishchenko2022asynchronous,feyzmahdavian2023asynchronous}.

\textbf{Notations.} $[n] \eqdef \{1, \dots, n\};$ $\N_0 \eqdef \{0, 1, 2, \dots\};$ $\norm{\cdot}$ is the standard Euclidean norm; $\inp{\cdot}{\cdot}$ is the standard dot product; $g = \cO(f):$ exist $C > 0$ such that $g(z) \leq C \times f(z)$ for all $z \in \mathcal{Z};$ $g = \Omega(f):$ exist $C > 0$ such that $g(z) \geq C \times f(z)$ for all $z \in \mathcal{Z};$ $g = \Theta(f):$ $g = \cO(f)$ and $g = \Omega(f);$ $g = \widetilde{\Theta}(f):$ the same as $g = \Theta(f)$ but up to logarithmic factors.

\subsection{Previous work}
\label{sec:previous_work}
\textbf{Oracle complexity.} In the classical optimization theory \citep{nemirovskij1983problem}, algorithms are compared in terms of \emph{oracle calls}. Assume that the number of workers is one and we work with nonconvex functions and Assumptions~\ref{ass:global_lipschitz_constant} and \ref{ass:stochastic_variance_bounded}.
It is well known \citep{arjevani2022lower,carmon2020lower} that the optimal oracle complexity is $\textstyle \cO\left(\nicefrac{L \Delta}{\varepsilon} + \nicefrac{\sigma^2 L \Delta}{\varepsilon^2}\right)$
to find $\bar{x} \in \R^d$ such that ${\mathbb{E}}[\norm{\nabla f(\bar{x})}^2] \leq \varepsilon.$
It is attained by the vanilla \algname{SGD} method: $x^{k+1} = x^{k} - \gamma \nabla f(x^k; \xi^k),$ where $\xi^k$ are i.i.d. random samples, $\Delta \eqdef f(x^0) - f^*,$ $x^0 \in \R^d$ is a starting point, and $\gamma = \Theta\left(\min\{\nicefrac{1}{L},\nicefrac{\varepsilon}{L \sigma^2}\}\right)$ is a step size. In the convex setting 
(Assumption~\ref{ass:convex}), 
the optimal oracle complexity is
$\textstyle \Theta\left(\nicefrac{\sqrt{L} R}{\sqrt{\varepsilon}} + \nicefrac{\sigma^2 R^2}{\varepsilon^2}\right)$ \citep{lan2020first,nemirovskij1983problem} to find $\bar{x} \in \R^d$ such that ${\mathbb{E}}[f(\bar{x})] - f(x_*) \leq \varepsilon,$ where $R \eqdef \norm{x^{0} - x_*}.$ In the $\mu$--strongly convex setting, the optimal complexity $\textstyle \widetilde{\Theta}\left(\nicefrac{\sqrt{L}}{\sqrt{\mu}} + \nicefrac{\sigma^2}{\mu^2 \varepsilon}\right)$ is to find $\bar{x} \in \R^d$ such that $\mathbb E [\norm{\bar{x} - x_*}^2] \leq \varepsilon$ (up to logarithmic factors).

\textbf{Oracle complexity with many workers.} Many works discovered oracle complexities with multiple workers.
\citet{arjevani2015communication,scaman2017optimal} analyze the heterogeneous convex setting and provide lower bounds when the workers are synchronized. \citet{lu2021optimal} consider the similar setup but in the nonconvex setting. \citet{arjevani2020tight} analyze settings where methods receive delayed stochastic gradients. \citet{woodworth2018graph} provide lower bounds for parallel setups with intermittent communications and delayed updates. The primary limitation of these results is the assumption that all workers have consistent computational performance, without accounting for individual delays, random lags, or variations in performance over time.

\textbf{Time complexity.} To address the problem of analyzing methods with workers having different computation capabilities and performances, \citet{mishchenko2022asynchronous} proposed to consider the \emph{fixed computation model}. In this model, it is assumed that
\begin{theorembox}
\centering
worker $i$ requires at most $\tau_i$ seconds to calculate one stochastic gradient. 
\end{theorembox}
Without loss of generality, we assume that the times are sorted: $\tau_1 \leq \dots \leq \tau_n.$ One of the most popular methods is \algname{Asynchronous SGD} \citep{lian2015asynchronous,zhang2015staleness,feyzmahdavian2016asynchronous,sra2016adadelay,dutta2018slow,stich2020error,wu2022delay,islamov2024asgrad,maranjyan2025ringmaster,maranjyan2026ringleader}. In the \emph{homogeneous setting}, \citet{mishchenko2022asynchronous,koloskova2022sharper,cohen2021asynchronous} showed that \mbox{\algname{Asynchronous SGD}} and \algname{Picky SGD} can provably improve the performance of the synchronized \algname{Minibatch SGD} method that does the steps 
$x^{k+1} = x^{k} - \nicefrac{\gamma }{n} \sum_{i=1}^n \nabla f(x^k; \xi^k_i),$ where $\gamma$ is a stepsize, $\xi^k_i$ are i.i.d. samples, and $\nabla f(x^k; \xi^k_i)$ are calculated in parallel in $n$ workers. \algname{Minibatch SGD} requires $\cO\left(\nicefrac{L \Delta}{\varepsilon} + \nicefrac{\sigma^2 L \Delta}{n \varepsilon^2}\right)$ iterations \citep{cotter2011better, goyal2017accurate, gower2019sgd} in the nonconvex setting. Moreover, \algname{Minibatch SGD} converges after
$
    \textstyle \cO\left(\max_{i \in [n]} \tau_i \times \left(\nicefrac{L \Delta}{\varepsilon} + \nicefrac{\sigma^2 L \Delta}{n \varepsilon^2}\right)\right)
$ seconds because it waits for the slowest worker with $\max_{i \in [n]} \tau_i$ in every iteration. \mbox{\algname{Asynchronous SGD}}, methods with the step $x^{k+1} = x^{k} - \nicefrac{\gamma^k}{n} \sum_{i=1}^n \nabla f(x^{k - \delta_k}; \xi^{k-\delta_k}_i)$ and $\delta_k$--delayed stochastic gradients, improve this time complexity to $\textstyle \cO(\left(\nicefrac{1}{n} \sum_{i=1}^n \nicefrac{1}{\tau_i}\right)^{-1}\left(\nicefrac{L \Delta}{\varepsilon} + \nicefrac{\sigma^2 L \Delta}{n \varepsilon^2}\right)).$ 

\textbf{Optimal time complexities in the heterogeneous and homogeneous settings.} Surprisingly, the time complexity can be further improved. In the nonconvex setup (under Assumptions~\ref{ass:global_lipschitz_constant}, and \ref{ass:stochastic_variance_bounded}), \citet{tyurin2023optimal} formalized the notion of time complexities and showed that the \emph{optimal time complexity} is
\begin{align}
  \label{eq:comp_homog}
  \squeeze T_{\textnormal{homog}} \eqdef \Theta\left(\min\limits_{m \in [n]} \left[\left(\frac{1}{m} \sum\limits_{i=1}^{m} \frac{1}{\tau_{i}}\right)^{-1} \left(\frac{L \Delta}{\varepsilon} + \frac{\sigma^2 L \Delta}{m \varepsilon^2}\right)\right]\right)
\end{align}
seconds \emph{in the homogeneous setup} to find an $\varepsilon$--stationary point, achieved by the \algname{Rennala SGD} method, where, without loss of generality, the times are sorted: $\tau_{1} \leq \dots \leq \tau_{n}.$ \emph{In the heterogeneous setup}, the optimal time complexity is
\begin{align}
  \label{eq:comp_heter}
  \squeeze T_{\textnormal{heter}} \eqdef \Theta\left(\tau_n \frac{L \Delta}{\varepsilon} + \left(\frac{1}{n} \sum\limits_{i=1}^{n} \tau_{i}\right) \frac{\sigma^2 L \Delta}{n \varepsilon^2}\right),
\end{align}
achieved by the \algname{Malenia SGD} method.

\textbf{A unifying perspective on \algname{Rennala SGD} and \algname{Malenia SGD}.}
Let us look closer to the \algname{Rennala SGD} and \algname{Malenia SGD} methods (see Algorithm~\ref{alg:alg_server_heterog}) that achieve the optimal time complexities \eqref{eq:comp_homog} and \eqref{eq:comp_heter} in the homogeneous and heterogeneous setting, accordingly. 
We now recall how they work. In every iteration, \algname{Rennala SGD} and \algname{Malenia SGD} ask all workers to calculate stochastic gradients asynchronously at \emph{the same iterate} $x^k$. Assume that worker $i$ has calculated $B_i^k$ stochastic gradients for all $i \in [n]$ at the iteration $k.$ Then the methods do the steps
\begin{algorithm}[t]
\caption{\algname{Malenia SGD} or \mbox{\algname{Rennala SGD}} when $w_i^k = \nicefrac{1}{B_i^k}$ or $w_i^k = \nicefrac{n}{\sum_{i=1}^n B_i^k},$ respectively}
  \label{alg:alg_server_heterog}
  \begin{algorithmic}[1]
  \STATE \textbf{Input:} point $x^0$, stepsize $\gamma$, parameter $S,$  weights $\{w_i^k\}$
  \FOR{$k = 0, 1, \dots, K - 1$}
  \STATE Ask all workers to calculate stochastic gradients at $x^k$; init $g^k_i = 0$ and $B_i^k = 0$ $\forall i \in [n]$
  \WHILE{$\left(\frac{1}{n} \sum_{i=1}^{n} (w_i^k)^2 B_i^k\right)^{-1} \leq \frac{S}{n}$}
  \STATE Wait for the next worker $j$
  \STATE Update $B_j^k = B_j^k + 1$
  \STATE Receive stochastic gradient $\nabla f_j(x^k;\xi^k_{j,B_j^k})$ and update $g^k_j = g^k_j + \nabla f_j(x^k;\xi^k_{j,B_j^k})$
  \STATE Ask this worker to calculate a stochastic gradient at $x^k$
  \ENDWHILE
  \STATE $g^k_{w} \eqdef \frac{1}{n} \sum_{i=1}^{n} w_i^k g^k_i =  \frac{1}{n} \sum_{i=1}^{n} w_i^k \sum_{j=1}^{B_i^k} \nabla f_i(x^k;\xi^k_{ij})$
  \STATE $x^{k+1} = x^k - \gamma g^k_{w}$
  \STATE Stop all the workers' calculations (or ignore the unfinished calculations in the next iterations)
  \ENDFOR
  \end{algorithmic}
\end{algorithm}
\begin{equation}
\begin{aligned}
  & \squeeze x^{k+1} = x^{k} - \gamma g^k_{\textnormal{\algname{R}}}, \quad g^k_{\textnormal{\algname{R}}} \eqdef \frac{1}{\sum_{i=1}^n B_i^k} \sum\limits_{i=1}^{n} \sum\limits_{j=1}^{B_i^k} \nabla f_i(x^k;\xi^k_{ij}) 
\end{aligned}
\tag{\algname{Rennala SGD}}
\label{eq:rennala}
\end{equation}
and 
\begin{equation}
\begin{aligned}
  &\squeeze x^{k+1} = x^{k} - \gamma g^k_{\textnormal{\algname{M}}}, \quad g^k_{\textnormal{\algname{M}}} \eqdef \frac{1}{n} \sum\limits_{i=1}^{n} \frac{1}{B_i^k}  \sum\limits_{j=1}^{B_i^k} \nabla f_i(x^k;\xi^k_{ij}),
\end{aligned}
\tag{\algname{Malenia SGD}}
\label{eq:malenia}
\end{equation}
accordingly. \algname{Rennala SGD} and \algname{Malenia SGD} ask all workers calculating stochastic gradients until $\frac{1}{n}\sum_{i=1}^n B_i^k > \nicefrac{S}{n}$ and $\left(\frac{1}{n} \sum_{i=1}^n \nicefrac{1}{B_i^k}\right)^{-1} > \nicefrac{S}{n}$ correspondingly, where $S$ is a parameter. Hence, both methods asynchronously collect and aggregate stochastic gradients to compute $g^k_{\textnormal{\algname{R}}}$ and $g^k_{\textnormal{\algname{M}}},$ and then perform a descent step.
However, the way the methods aggregate is both different and important. It turns out the variance of the \algname{Rennala SGD}'s update is smaller. Indeed, one can easily show that
\begin{align*}
  \squeeze \Exp{\norm{g^k_{\textnormal{\algname{R}}} - \Exp{g^k_{\textnormal{\algname{R}}}}}^2} \leq 
  \frac{\sigma^2}{n} \left(\frac{1}{n} \sum\limits_{i=1}^n B_i^k\right)^{-1} \textnormal{ and } \Exp{\norm{g^k_{\textnormal{\algname{M}}} - \Exp{g^k_{\textnormal{\algname{M}}}}}^2} \leq \frac{\sigma^2}{n} \left(\frac{n}{\sum_{i=1}^n \frac{1}{B_i^k}}\right)^{-1}.
\end{align*}
Thus, the variance of \ref{eq:rennala} improves with the \emph{arithmetic mean} of $B_i^k,$ while the variance of \ref{eq:malenia} improves with the \emph{harmonic mean} of $B_i^k,$ which can be much smaller. Why wouldn't we use \ref{eq:rennala} in all scenarios if it is better? Because $g^k_{\textnormal{\algname{R}}}$ is biased if $\{f_i\}$ are non-homogeneous. In general, $\Exp{g^k_{\textnormal{\algname{R}}}} \neq \nabla f(x^k),$ while it is always true that $\Exp{g^k_{\textnormal{\algname{M}}}} = \nabla f(x^k).$
Both methods can be generalized into
  $\squeeze x^{k+1} = x^{k} - \gamma g^k_{w},$ $g^k_{w} \eqdef \frac{1}{n} \sum_{i=1}^{n} w_i^k \sum_{j=1}^{B_i^k} \nabla f_i(x^k;\xi^k_{ij}),$
where the weights $\{w_i^k\}$ are free parameters. If we take $w_i^k = \nicefrac{n}{\sum_{i=1}^n B_i^k}$ for all $i \in [n],$ we get \ref{eq:rennala} with small variance. If we take $w_i^k = \nicefrac{1}{B_i^k},$ we get \ref{eq:malenia} with high variance but with an unbiased estimator. The weights enable interpolation between the methods. 

\textbf{Difference between the two settings.} Using the inequality of arithmetic and harmonic means, one can easily show that $T_{\textnormal{homog}} \leq T_{\textnormal{heter}}$ (ignoring constant factors). At the same time, the gap between the complexities can be arbitrarily huge. Indeed, when the performance $\tau_1$ of the fastest worker tends to $0,$ one can easily show that $T_{\textnormal{homog}} \to 0$ and $T_{\textnormal{heter}} \to \Theta\left(\nicefrac{\tau_n L \Delta}{\varepsilon} + \left(\frac{1}{n} \sum_{{\color{red} i = 2}}^{n} \tau_{i}\right) \nicefrac{\sigma^2 L \Delta}{n \varepsilon^2}\right),$ and $T_{\textnormal{heter}}$ improves by at most $\sum_{i=1}^{n} \tau_{i} / \sum_{i=2}^{n} \tau_{i} \leq 2.$ While the improvement in the homogeneous setup is $\infty.$ Consider another example when the performance $\tau_n$ of the slowest worker (straggler) tends to $\infty.$ Then $T_{\textnormal{heter}} \to \infty$ and $T_{\textnormal{homog}} \to \Theta(\min_{m \in [{\color{red} n - 1}]} [\left(\nicefrac{1}{m} \sum_{i=1}^{m} \nicefrac{1}{\tau_{i}}\right)^{-1} \left(\nicefrac{L \Delta}{\varepsilon} + \nicefrac{\sigma^2 L \Delta}{m \varepsilon^2}\right)]),$ so the complexity $T_{\textnormal{homog}}$ is robust to stragglers unlike $T_{\textnormal{heter}}.$

\textbf{Arbitrarily computation dynamics.} The previous discussion explain that a significant gap appears between homogeneous and heterogeneous problems under the fixed computation model. This ``arithmetic mean vs harmonic mean gap'' was also observed in \citep{tyurin2024tight}, where the author generalizes the fixed computation model to the \emph{universal computation model}, 
accounting for potential disruptions caused by hardware or network delays, and any variations in computation speeds. 
For simplicity, in this work, we will continue working with the fixed computation model, but we also show how our final results translate to the universal computation model in Section~\ref{sec:arbitrarily}.

\textbf{Convex world.} When we want to find a point $\bar{x}$ such that $\Exp{f(\bar{x})} - f^* \leq \varepsilon$ in the convex setup, the gap is similar. The optimal time complexity \emph{in the homogeneous setup} is 
\begin{align}
  \label{eq:comp_homog_convex}
  \squeeze \Theta\left(\min\limits_{m \in [n]} \left[\left(\frac{1}{m} \sum\limits_{i=1}^{m} \frac{1}{\tau_{i}}\right)^{-1} \left(\frac{\sqrt{L} R}{\sqrt{\varepsilon}} + \frac{\sigma^2 R^2}{m \varepsilon^2}\right)\right]\right)
\end{align}
seconds \citep{tyurin2023optimal}. While the optimal time complexity \emph{in the heterogeneous setup} is
\begin{align}
  \label{eq:comp_heter_convex}
  \squeeze \Theta\left(\tau_n \frac{\sqrt{L} R}{\sqrt{\varepsilon}} + \left(\frac{1}{n} \sum\limits_{i=1}^{n} \tau_{i}\right) \frac{\sigma^2 R^2}{n \varepsilon^2}\right)
\end{align}
seconds under Assumptions~\ref{ass:global_lipschitz_constant}, \ref{ass:convex}, and \ref{ass:stochastic_variance_bounded} (\textbf{our new contribution}, Theorem~\ref{thm:first_lower_bound}; the final puzzle piece needed to reveal the systematic gap between the two settings). Both complexities are achieved by the accelerated versions of \algname{Rennala SGD} and \algname{Malenia SGD} accordingly.

\textbf{Strongly convex world.} Assume additionally that the function $f$ is $\mu$--strongly convex.
Using reduction \citep{woodworth2016tight}, up to logarithmic factors, we can obtain the optimal time complexity 
\begin{align}
  \label{eq:comp_homog_convex_strongly}
  \squeeze \widetilde{\Theta}\left(\min\limits_{m \in [n]} \left[\left(\frac{1}{m} \sum\limits_{i=1}^{m} \frac{1}{\tau_{i}}\right)^{-1} \left(\sqrt{\frac{L}{\mu}} + \frac{\sigma^2}{m \varepsilon \mu}\right)\right]\right)
\end{align}
in the homogeneous setting and the optimal time complexity
\begin{align}
  \label{eq:comp_heter_convex_strongly}
  \squeeze \widetilde{\Theta}\left(\tau_n \sqrt{\frac{L}{\mu}} + \left(\frac{1}{n} \sum\limits_{i=1}^{n} \tau_{i}\right) \frac{\sigma^2}{n \varepsilon \mu}\right)
\end{align}
in the heterogeneous setting when we want to find a point $\bar{x}$ such that $\Exp{f(\bar{x})} - f^* \leq \varepsilon.$ 
Here we also observe a large gap between the settings. Note that the complexities \eqref{eq:comp_heter}, \eqref{eq:comp_heter_convex}, and \eqref{eq:comp_heter_convex_strongly} can only be improved under additional assumptions because they are optimal.
\begin{theorembox}
\begin{quote}
    \textbf{Main question:} Having the systematic gap between the homogeneous and heterogeneous setups, the goal of this work is to identify theoretical assumptions that are as weak as possible to improve the results of asynchronous methods in heterogeneous scenarios. Under which assumptions can we improve the dependence on the arithmetic mean of $\{\tau_i\}$ (see \eqref{eq:comp_heter}, \eqref{eq:comp_heter_convex}, and \eqref{eq:comp_heter_convex_strongly}) to the dependence on the harmonic mean of $\{\tau_i\}$ (see \eqref{eq:comp_homog}, \eqref{eq:comp_homog_convex}, and \eqref{eq:comp_homog_convex_strongly})?
    Right now, the only possible way is to assume that the functions $\{f_i\}$ are equal---an assumption we clearly want to avoid in the heterogeneous setting. Is there any chance to relax this assumption?
\end{quote}
\end{theorembox}

\subsection{Contributions}
To the best of our knowledge, this is the first work to address the main question in any setting; our analysis considers the convex setting under standard Assumptions~\ref{ass:global_lipschitz_constant}, \ref{ass:lipschitz_constant}, \ref{ass:convex}, \ref{ass:stochastic_variance_bounded}, and \ref{ass:pl_global}.

\textbf{Analysis of first- and second-order similarity.} First, we consider the celebrated \emph{first- and second-order similarity} and, surprisingly, prove that even under these assumptions—no matter how close the functions $\{f_i\}$ are—\emph{any randomized algorithm} cannot converge before $\Omega\Big(\Big(\frac1n\sum\limits_{i=1}^n\tau_i\Big) \frac{\sigma^2}{n\mu^2\varepsilon}\Big)$ seconds for small $\varepsilon$ (Theorem~\ref{thm:divergence_first}). Thus, it is infeasible to break the dependence on the arithmetic mean of $\{\tau_i\}$ under these assumptions.

\textbf{Investigation of the interpolation assumption.} Inspired by Theorem~\ref{thm:divergence_first}, which provides a construction with local functions having different minimizers, we decided to go in another direction and consider the \emph{interpolation} assumption. Thus, we introduce two additional assumptions, strong interpolation and the local Polyak-\L ojasiewicz condition, and prove that it is impossible to drop either of these assumptions for improvement (Theorems~\ref{thm:divergence_2} and \ref{thm:divergence_3}).

\textbf{Bridging the gap.} By identifying this minimal set of assumptions, we derive a new time complexity result that matches the dependence on worker computation times in the best-known bound in the \emph{homogeneous} setting (Theorem~\ref{thm:main_theorem}), but without requiring the functions $f_i$ to be identical. Our theoretical results are validated numerically in Section~\ref{sec:exp}.

\emph{To bridge the gap in Section~\ref{sec:main}, we need to introduce Assumptions~\ref{ass:inter_strong} and \ref{ass:pl_condition}. However, our primary goal was to illustrate and prove that these assumptions are indeed necessary. Merely stating the assumptions might not be convincing; this is why the central part of our paper investigates different assumptions and shows that most of them do not allow us to bridge the gap. 
While previous work noted the existence of the gap, our contribution goes further by systematically investigating which assumptions are sufficient and which are insufficient to eliminate it.}
\section{First-Order and Second-Order Similarity Don't Help}
\label{sec:first_second}
The main problem with the \emph{arithmetic mean dependence} in the heterogeneous setting is that this setting considers a worst-case scenario with arbitrarily heterogeneous functions.
Due to the fact that \algname{Malenia SGD} is optimal, we have to introduce \emph{assumptions} to obtain faster convergence. One of the most popular assumptions in the literature is \emph{first-order and second-order similarity of the functions} \citep{arjevani2015communication,szlendak2021permutation,mishchenko2022asynchronous}:
\begin{assumption}[First-Order Similarity]
    \label{ass:first}
    The functions $f_i$ satisfy 
        $\max_{i,j \in [n]}\norm{\nabla f_i(x) - \nabla f_j(x)}^2 \leq \delta_1$
    for all $x \in \R^d$ for some $\delta_1 \geq 0$. It implies $\frac{1}{n} \sum_{i=1}^{n}\norm{\nabla f_i(x) - \nabla f(x)}^2 \leq \delta_1$ for all $x \in \R^d.$
\end{assumption}
\begin{assumption}[Second-Order Similarity]
    \label{ass:second}
    The functions $f_i$ satisfy 
        $\max_{i,j \in [n]} \norm{\nabla^2 f_i(x) - \nabla^2 f_j(x)}^2 \leq \delta_2$
    for all $x \in \R^d$ for some $\delta_2 \geq 0$. It implies $\frac{1}{n} \sum_{i=1}^{n}\norm{\nabla^2 f_i(x) - \nabla^2 f(x)}^2 \leq \delta_2$ for all $x \in \R^d.$
\end{assumption}
One might expect that when both $\delta_1$ and $\delta_2$ are small, it would be possible to exploit the similarity and design a method with smaller variance and better dependence on $\{\tau_i\}.$
Surprisingly, this is not the case: for any $\delta_1 > 0$ and $\delta_2 \geq 0$, one can construct a problem for which the convergence speed of \algname{Malenia SGD} (Theorem~\ref{thm:malenia}) cannot be improved, up to logarithmic factors, in small-$\varepsilon$ regimes:
\begin{restatable}[Lower Bound]{theorem}{FIRSTORDER}
  \label{thm:divergence_first}
  Consider stochastic gradients $\nabla f_i(x;\xi_i) = \nabla f_i(x) + \xi_i e_i$ with $\xi_i \sim \mathcal{N}(0,\sigma^2)$ for all $i \in [n]$, and $x \in \R^n.$ Consider any randomized algorithm that has access only to the stochastic gradients (randomized method), which starts at $x^0 = 0$, under the fixed computation model and any $R, \mu,\beta,\sigma, \varepsilon > 0$ such that $0 < \varepsilon \leq \frac{c\beta^2}{\mu^2n^2}$ and $R^2 \geq \frac{\beta^2}{\mu^2 n},$ where $c > 0$ is a universal constant. For any time budget
    $\textstyle t \leq c_0 \left(\frac1n\sum\limits_{i=1}^n\tau_i\right)
        \frac{\sigma^2}
             {n\mu^2\varepsilon},$
  where $c_0$ is a universal constant, there exist $f_i(x) \,:\, \R^n \to \R$ such that
  $f_i(x) = \frac{\mu}{2} \norm{x}^2 - \beta \varphi_i \inp{x}{e_i}$ and  
  $\varphi_i \in [-1, 1].$ Assumptions~\ref{ass:global_lipschitz_constant}, \ref{ass:lipschitz_constant}, \ref{ass:convex}, \ref{ass:stochastic_variance_bounded}, and \ref{ass:pl_global} hold. Moreover, Assumption~\ref{ass:first} (the first-order similarity) is satisfied with $\delta_1 = 2 \beta^2,$ and Assumption~\ref{ass:second} (the second-order similarity) is satisfied with $\delta_2 = 0,$ $\norm{x^0 - x^*}^2 \leq R^2,$ and the method cannot produce a point $\bar{x}$ such that $\Exp{\norm{\bar{x}-x^*}^2}\leq\varepsilon$ within $t$ seconds, where $x^*$ is the minimizer of $f.$
\end{restatable}
Hence, for any small $\delta_1 > 0$ and $\delta_2 \geq 0$, the convergence speed cannot be improved over that of \algname{Malenia SGD} up to logarithmic factors in small-$\varepsilon$ regimes (compare to Theorem~\ref{thm:malenia}). Due to the construction in Theorem~\ref{thm:divergence_first}, we can choose any $\beta > 0$, and hence any $\delta_1 > 0$. No matter how close the functions are to each other, the lower bound does not allow us to break the pessimistic time complexity. In view of this, additional assumptions about the first- and second-order similarity will not help to improve the time complexity of \algname{Malenia SGD}.
\begin{remark}
    For the construction in Theorem~\ref{thm:divergence_first}, we can also show that $\norm{\nabla f_i(x)}^2 \leq 2 \norm{\nabla f(x)}^2 + 2 \beta^2 d$ for all $i \in [n],$ which corresponds to the $\rho$--\emph{strong growth} condition when $\beta = 0$ and $\rho = 2$ \citep{schmidt2013fast}. Since Theorem~\ref{thm:divergence_first} holds for all $\beta > 0,$ we have proved the result for a ``slightly'' broader class of problems and have ``almost'' established that, even under the \emph{strong growth} condition, the convergence speed of \algname{Malenia SGD} cannot be improved up to logarithmic factors. Whether a similar result holds for the class of problems satisfying $\max_{i \in [n]}\norm{\nabla f_i(x)}^2 \leq 2 \norm{\nabla f(x)}^2$ for all $x \in \R^d$ remains an important open research question.
\end{remark}
\begin{takeawaybox}
   \label{ta:first_sceond}
   Even with first-order and second-order similarity, for any randomized algorithm, there is still no hope of improving upon the convergence speed of \algname{Malenia SGD}, up to logarithmic factors, in small-$\varepsilon$ regimes.
\end{takeawaybox}

\section{Understanding the Gap via Interpolation Assumptions}
\label{sec:challenges}
Looking at Takeaway~\ref{ta:first_sceond}, we see that a different similarity assumption is required to close the gap between the heterogeneous and homogeneous results. 
\begin{table*}[!t]
    \caption{The summary of our results and the time complexities (up to logarithmic factors) to get a point $\bar{x}$ such that $\mathbb{E}[\norm{\bar{x} - x_*}^2] \leq \varepsilon$ under the \emph{fixed computation model} (worker $i$ requires at most $\tau_i$ seconds to calculate one stochastic gradient; $\tau_1 \leq \dots \leq \tau_n$) and Assumptions~\ref{ass:global_lipschitz_constant}, \ref{ass:lipschitz_constant}, \ref{ass:convex}, \ref{ass:stochastic_variance_bounded}, and \ref{ass:pl_global}, where $\bar{x}_*$ is the closest solution to $\bar{x}.$ The table compares methods in the fully heterogeneous setting and lists the extra assumptions the methods require to work.}
    \label{table:complexities}
    \centering 
    \begin{adjustbox}{width=1.0\columnwidth,center}  
    \begin{threeparttable}  
      \begin{tabular}[t]{cccccc}
  \toprule
     \bf  Method & \bf Time Complexity Guarantees & Additional Assumptions \\
       \midrule
       \makecell{\algname{Minibatch SGD}} & $\tau_n \left(\frac{L}{\mu} + \frac{\sigma^2}{n \varepsilon \mu^2}\right)$ & --- \\
       \midrule
       \makecell{\algname{Asynchronous SGD} \\ \citep{mishchenko2022asynchronous}} & $\left(\frac{1}{n} \sum\limits_{i=1}^{n} \frac{1}{\tau_{i}}\right)^{-1} \left(\frac{L}{\mu} + \frac{\sigma^2}{n \varepsilon \mu^2}\right)$ & \makecell{$\{f_i\}$ are equal \\ $\mu$--strong convexity} \\
       \midrule
       \makecell{\algname{Malenia SGD} \\ \citep{tyurin2023optimal} \\ (Theorem~\ref{thm:malenia})} & $\tau_n \frac{L}{\mu} + \left(\frac{1}{n} \sum\limits_{i=1}^{n} \tau_{i}\right) \frac{\sigma^2}{n \varepsilon \mu^2}$ & --- \\
       \midrule
       \makecell{\algname{Rennala SGD} \\ \citep{tyurin2023optimal} \\ (Theorem~\ref{thm:rennala})} & $\min\limits_{m \in [n]} \left[\left(\frac{1}{m} \sum\limits_{i=1}^{m} \frac{1}{\tau_{i}}\right)^{-1} \left(\frac{L}{\mu} + \frac{\sigma^2}{m \varepsilon \mu^2}\right)\right]$ & $\{f_i\}$ are equal \\
   \toprule
    \multicolumn{3}{c}{\makecell{\bf Lower Bounds (new results)}} \\
      \toprule
       \multicolumn{3}{c}{\makecell{Under the first-order and second-order similarity, the following results state that \\ it is infeasible to improve \algname{Malenia SGD} in small-$\varepsilon$ regimes:}} \\
       \midrule
       \makecell{Any randomized method \\ (Theorem~\ref{thm:divergence_first})} & $\geq \left(\frac{1}{n} \sum\limits_{i=1}^{n} \tau_{i}\right) \frac{\sigma^2}{n \varepsilon \mu^2}$ & \makecell{Assumptions~\ref{ass:first} and \ref{ass:second} \\ (first-order and second-order similarity don't help)} \\
      \toprule
       \multicolumn{3}{c}{\makecell{Under weak interpolation, the following result states that it is infeasible to improve \algname{Malenia SGD} in small-$\varepsilon$ regimes:}} \\
       \midrule
       \makecell{Any randomized method \\ (Theorem~\ref{thm:allneeded})} & $\geq \left(\frac{1}{n} \sum\limits_{i=1}^{n} \tau_{i}\right) \frac{\sigma^2}{n \varepsilon \mu^2}$ & \makecell{Assumption~\ref{ass:inter}} \\
      \toprule
       \multicolumn{3}{c}{\makecell{The following results state that any randomized method can not improve \algname{Malenia SGD} for small $\varepsilon$ if we discard Assumption~\ref{ass:inter_strong} or \ref{ass:pl_condition}:}} \\
       \midrule
       \makecell{Any randomized method \\ (Theorem~\ref{thm:divergence_2})} & $\geq \left(\frac{1}{n} \sum\limits_{i=1}^{n} \tau_{i}\right) \frac{\sigma^2}{n \varepsilon \mu^2}$ & \makecell{Assumptions~\ref{ass:inter} and \ref{ass:pl_condition} \\ (weak interpolation is not enough)} \\
       \midrule
       \makecell{Any randomized method \\ (Theorem~\ref{thm:divergence_3})} & $\geq \left(\frac{1}{n} \sum\limits_{i=1}^{n} \tau_{i}\right) \frac{\sigma^2}{n \varepsilon \mu^2}$ & \makecell{Assumption~\ref{ass:inter_strong}} \\
      \toprule
       \multicolumn{3}{c}{\bf Upper Bound (new result)} \\
       \toprule
       \multicolumn{3}{c}{\makecell{The following results state that under Assumption~\ref{ass:inter_strong} and \ref{ass:pl_condition} it is possible to improve \algname{Malenia SGD}:}} \\
       \midrule
       \makecell{\algname{Rennala SGD} \\ (Theorem~\ref{thm:main_theorem})} & $\min\limits_{m \in [n]} \left[\left(\frac{1}{m} \sum\limits_{i=1}^{m} \frac{1}{\tau_{i}}\right)^{-1} \left(\frac{L_{\max}}{\mu} + \frac{\sigma^2}{m \varepsilon \mu^2}\right)\right]$ & \makecell{Assumptions~\ref{ass:inter_strong} and \ref{ass:pl_condition} \\ (weaker than the equality of functions $\{f_i\}$)} \\
      \bottomrule
      \end{tabular}
    \end{threeparttable}
    \end{adjustbox}
  \end{table*}
Recall Theorem~\ref{thm:divergence_first}. The local minima of the functions ${f_i}$ are not the same. This motivates us to explore an alternative assumption known as the \emph{interpolation} assumption \citep{vaswani2019painless}. This assumption provides another way to capture the similarity among the functions $f_i$ by requiring that they share the same set of minimizers as the function $f$.
\begin{assumption}[Weak Interpolation]
    \label{ass:inter}
    If $x^*$ is a minimizer of $f$, that is, $\nabla f(x^*) = 0$, then $x^*$ is also a minimizer of each $f_i$ for all $i \in [n].$
\end{assumption}
Interpolation is a property of the solutions of ${f_i},$ whereas the heterogeneity assumptions, Assumptions~\ref{ass:first} and~\ref{ass:second}, concern the gradients and Hessians. These are different characteristics of ${f_i}$ (see Remark~\ref{rmk:new}).
Assumption~\ref{ass:inter} is considered practical in modern optimization literature, as there is evidence that it holds for large deep learning models \citep{zou2019improved,zhang2021understanding}. However, as we show next, this assumption alone is not sufficient to achieve improved time complexity, leading to yet another pessimistic result:
\begin{restatable}[Lower Bound]{theorem}{ALLNEEDED}
  \label{thm:allneeded}
  Consider any randomized method under the fixed computation model and assume that $n \geq 2.$ Let us fix any $\varepsilon, L_{\max}, R, \mu, \sigma^2 > 0$ such that $\mu < L_{\max} / (2 n),$ $\varepsilon < 0.01,$ and $R > 10.$ For any time budget $\textstyle t \leq c_0 \left(\frac{1}{n}\sum\limits_{i=1}^n \tau_i\right) \frac{\sigma^2}{\varepsilon n \mu^2},$
  where $c_0$ is a universal constant, there exist functions $\{f_i\}$ and stochastic gradients $\{\nabla f_i(\cdot;\cdot)\}$ such that $\{f_i\}$ satisfy Assumptions~\ref{ass:lipschitz_constant}, \ref{ass:convex}, and \ref{ass:inter}, $f$ satisfies Assumptions~\ref{ass:global_lipschitz_constant} and \ref{ass:pl_global} with $L = L_{\max}$, $\{\nabla f_i(\cdot;\cdot)\}$ satisfy Assumption~\ref{ass:stochastic_variance_bounded} such that the method cannot find $\varepsilon$--solution in terms of distances to the solution set after $t$ seconds, when the method starts at a point in a distance less or equal to $R$ to the closest solution.
\end{restatable}
Thus, even under Assumption~\ref{ass:inter}, we can not improve the arithmetic mean dependence on $\{\tau_i\}.$ 
\begin{takeawaybox}
   \label{ta:interp}
   Using the weak interpolation assumption, which captures the similarity of the functions in a different way compared to first-order and second-order similarity, it is still infeasible to improve the pessimistic dependence on $\{\tau_i\}$ achieved by \algname{Malenia SGD} using any randomized algorithm.
\end{takeawaybox}
\subsection{Strong interpolation and local P\L\,condition are both required}
\label{sec:important_assumptions}
Once again, we need to go deeper and introduce additional assumptions to break the lower bound from Theorem~\ref{thm:allneeded}. To further investigate the problem, we now turn to two related assumptions.
\begin{assumption}[Strong Interpolation]
    \label{ass:inter_strong}
    For all $i \in [n],$ a point $x^*$ is a minimizer of $f$, that is, $\nabla f(x^*) = 0$, \emph{if and only if} it is also a minimizer of $f_i.$
\end{assumption}
This assumption is clearly stronger than the weak interpolation assumption since it requires all the functions to share the set of minimizers (see Remark~\ref{rmk:new2}).
\begin{assumption}[Local Polyak-\L ojasiewicz condition]
  \label{ass:pl_condition}
  There exists $\mu$ such that $\norm{\nabla f_i(x)}^2 \geq 2 \mu \left(f_i(x) - f_i^*\right)$ for all $x \in \R^d$ and for all $i \in [n],$ where $f_i^*$ is the finite optimal function value of $f_i.$
\end{assumption}
This assumption, unlike Assumption~\ref{ass:pl_global}, requires each function to satisfy P\L\,condition. 
\begin{remark}
\label{rmk:new}
The similarity and interpolation assumptions are neither disjoint nor does one imply the other. For example, consider $f_i(x)=\frac{1}{2}x^2+c_i(1-\cos x),$ $0\leq c_i<1.$ These functions have the same unique minimizer $x^\star=0$ and are strongly convex and smooth, and thus satisfy strong interpolation and the local PŁ condition. At the same time,
$
|f_i'(x)-f_j'(x)|=|c_i-c_j||\sin x|,$
$|f_i''(x)-f_j''(x)|=|c_i-c_j||\cos x|,
$
so the first- and second-order similarity assumptions also hold. Conversely, the construction in Theorem~\ref{thm:divergence_first} satisfies the similarity assumptions while interpolation fails. In the other direction, the functions $f_i(x)= a_i x^2 / 2,$$a_i>0,$
have the same minimizer and satisfy the local PŁ condition, whereas for $a_i\neq a_j$ their gradient difference $|(a_i-a_j)x|$ is unbounded. Thus, interpolation-type and similarity assumptions describe different, overlapping forms of heterogeneity.
\end{remark}
It turns out again that if we do not assume \emph{both} Assumption~\ref{ass:inter_strong} and Assumption~\ref{ass:pl_condition}, then it is infeasible for any randomized method to get a time complexity faster than in \algname{Malenia SGD} (Theorem~\ref{thm:malenia}) for $\varepsilon$ small enough. This statement is formalized in the following two theorems.

\begin{restatable}[Lower Bound]{theorem}{WEAKNEEDED}
  \label{thm:divergence_2}
  Consider any randomized method under the fixed computation model and assume that $n \geq 2.$ Let us fix any $\varepsilon, L_{\max}, R, \mu, \sigma^2 > 0$ such that $\mu < L_{\max} / (2 n),$ $\varepsilon < 0.01,$ and $R > 10.$ For any time budget $\textstyle t \leq c_0 \left(\frac{1}{n}\sum\limits_{i=1}^n \tau_i\right) \frac{\sigma^2}{\varepsilon n \mu^2},$
  where $c_0$ is a universal constant, there exist functions $\{f_i\}$ and stochastic gradients $\{\nabla f_i(\cdot;\cdot)\}$ such that $\{f_i\}$ satisfy Assumptions~\ref{ass:lipschitz_constant}, \ref{ass:convex}, \ref{ass:inter}, and \ref{ass:pl_condition} \textbf{(Assumption~\ref{ass:inter_strong} is not imposed and may or may not hold)}, $f$ satisfies Assumptions~\ref{ass:global_lipschitz_constant} and \ref{ass:pl_global} with $L = L_{\max}$, $\{\nabla f_i(\cdot;\cdot)\}$ satisfy Assumption~\ref{ass:stochastic_variance_bounded} such that the method cannot find $\varepsilon$--solution in terms of distances to the solution set after $t$ seconds, when the method starts at a point in a distance less or equal to $R$ to the closest solution.
\end{restatable}

\begin{restatable}[Lower Bound]{theorem}{PLNEEDED}
  \label{thm:divergence_3}
  Consider any randomized method under the fixed computation model and assume that $n \geq 2.$ Let us fix any $\varepsilon, L_{\max}, R, \mu, \sigma^2 > 0$ such that $\mu < L_{\max} / (2 n),$ $\varepsilon < 0.01,$ and $R > 10.$ For any time budget $\textstyle t \leq c_0 \left(\frac{1}{n}\sum\limits_{i=1}^n \tau_i\right) \frac{\sigma^2}{\varepsilon n \mu^2},$
  where $c_0$ is a universal constant, there exist functions $\{f_i\}$ and stochastic gradients $\{\nabla f_i(\cdot;\cdot)\}$ such that $\{f_i\}$ satisfy Assumptions~\ref{ass:lipschitz_constant}, \ref{ass:convex}, \ref{ass:inter}, and \ref{ass:inter_strong} \textbf{(Assumption~\ref{ass:pl_condition} is not imposed and may or may not hold with parameter $\mu$)}, $f$ satisfy Assumptions~\ref{ass:global_lipschitz_constant} and \ref{ass:pl_global} with $L = L_{\max}$, $\{\nabla f_i(\cdot;\cdot)\}$ satisfy Assumption~\ref{ass:stochastic_variance_bounded},
  such that the method cannot find $\varepsilon$--solution in terms of distances to the solution set after $t$ seconds, when the method starts at a point in a distance less or equal to $R$ to the closest solution.
\end{restatable}
\begin{takeawaybox}
   \label{ta:interp_new}
   Even when the weak interpolation assumption is combined with only one of Assumptions~\ref{ass:inter_strong} and \ref{ass:pl_condition}, we still obtain only the arithmetic mean dependence on $\{\tau_i\}.$
\end{takeawaybox}
Once we drop either Assumption~\ref{ass:inter_strong} or Assumption~\ref{ass:pl_condition}, it becomes possible to construct a ``bad'' function (see the proof of theorems) that provides no room for any randomized method to improve. 
\subsection{Finally bridging the gap}
\label{sec:main}
However, if assume that both Assumption~\ref{ass:inter_strong} and Assumption~\ref{ass:pl_condition} hold, then, finally, we can proof the convergence with harmonic-like dependence on $\{\tau_i\}:$ 
\begin{restatable}[Upper Bound]{theorem}{MAINTHEOREM}
  \label{thm:main_theorem}
  Let Assumptions~\ref{ass:lipschitz_constant}, \ref{ass:convex}, \ref{ass:stochastic_variance_bounded}, \ref{ass:inter_strong}, \ref{ass:pl_condition} hold\footnote{It is well-know that Assumption~\ref{ass:lipschitz_constant} implies Assumption~\ref{ass:global_lipschitz_constant}. In Section~\ref{sec:aux_imply}, we prove that Assumptions~\ref{ass:convex}, \ref{ass:inter_strong} and Assumption~\ref{ass:pl_condition} with constant $\mu$ imply Assumption~\ref{ass:pl_global} with constant $\mu / 4$.}. 
  We choose $w_i^k = \nicefrac{n}{\sum_{i=1}^n B_i^k}$ for all $k \geq 0, i \in [n]$ in Algorithm~\ref{alg:alg_server_heterog} (reduces to \mbox{\algname{Rennala~SGD}}). 
  We take $\gamma = \nicefrac{1}{L_{\max}},$ $S = \nicefrac{4 \sigma^2}{\mu L_{\max} \varepsilon},$ and run \mbox{\algname{Rennala SGD}} for $k \geq \Omega\left(\frac{L_{\max}}{\mu} \log \frac{R^2}{\varepsilon}\right)$ iterations, then $\Exp{\norm{x^{k+1} - x_*^{k+1}}^2} \leq \varepsilon,$ where $x_*^{k+1}$ is the closest solution to $x^{k+1}.$ Moreover, under the fixed computation model, the method requires 
    $\squeeze \cO\left(\min\limits_{m \in [n]} \left[\left(\frac{1}{m} \sum\limits_{i=1}^{m} \frac{1}{\tau_{i}}\right)^{-1} \left(\frac{L_{\max}}{\mu} + \frac{\sigma^2}{m \varepsilon \mu^2}\right)\right]\log \frac{R^2}{\varepsilon}\right)$
  seconds.
\end{restatable}
Under weaker assumptions, without requiring the equality of the functions $\{f_i\},$ this theorem yields time complexity guarantees with a ``harmonic''-like dependence on the times $\{\tau_i\}$ for the \algname{Rennala SGD} method, improving upon the previous theoretical results in Theorem~\ref{thm:rennala} and \citep{tyurin2023optimal}.
Notice that the method in Theorem~\ref{thm:main_theorem} is still biased because 
 $\Exp{\sum_{i=1}^{n} \sum_{j=1}^{B_i^k} \nabla f_i(x;\xi^k_{ij}) / \sum_{i=1}^n B_i^k} \neq \nabla f(x)$
in general. That said, we can successfully prove the theorem under this constraint. One of the primary reasons for this is the right choice of \emph{convergence metric}. Initially, we aimed to analyze the biased gradient estimator in terms of function values and gradient norms, trying to prove that the method returns a point $\bar{x}$ such that $\mathbb{E}[f(\bar{x})] - f^* \leq \varepsilon$ or $\mathbb{E}[\norm{\nabla f(\bar{x})}^2] \leq \varepsilon$. However, the more appropriate approach is to show $\mathbb{E}[\norm{\bar{x} - x_*}^2] \leq \varepsilon$. Using this convergence metric allows us to analyze the biased gradient estimator. This observation can be important on its own. Note that we can get convergence in terms of $\mathbb{E}[f(\bar{x})] - f^* \leq \varepsilon$ using $L$--smoothness, but the result would be loose. One interesting observation is that we do not observe a regime where any other method or strategy improves upon both \algname{Malenia SGD} and \algname{Rennala SGD}.
\begin{takeawaybox}
   \label{ta:final}
   Improving the pessimistic dependence in \algname{Malenia SGD} is possible with \algname{Rennala SGD} and the additional assumptions, Assumption~\ref{ass:inter_strong} and Assumption~\ref{ass:pl_condition}, in convex optimization.
\end{takeawaybox}
\begin{remark}
\label{rmk:new2}
Theorem~\ref{thm:main_theorem} is proved under the strong interpolation assumption. This assumption is essential for our result: even relaxing strong interpolation to weak interpolation is insufficient to improve the pessimistic time complexity achieved by \algname{Malenia SGD}. At the same time, strong interpolation does not require the local functions ${f_i}$ to be identical. For example, consider the least-squares problems $f_i(x)=\frac{1}{2}\|\mA_i(x-x^\star)\|^2,$ where the matrices $\mA_i$ can be different but have a common kernel $N$. Then $\operatorname{argmin} f_i=x^\star+N$
for every worker $i$, so the strong interpolation assumption holds. Thus, this setting provides a concrete machine-learning example with heterogeneous, non-identical local functions 
covered by Theorem~\ref{thm:main_theorem}.
\end{remark}
\subsection{Extension to nonconvex and general convex optimization}
Although in this paper we focus on the convergence metric $\mathbb{E}[\|x^k - x_*\|^2]$ under Assumptions~\ref{ass:global_lipschitz_constant}, \ref{ass:lipschitz_constant}, \ref{ass:convex}, \ref{ass:stochastic_variance_bounded}, and \ref{ass:pl_global}, we briefly sketch how the constructions behind Theorems~\ref{thm:divergence_first}, \ref{thm:allneeded}, \ref{thm:divergence_2}, and \ref{thm:divergence_3} may extend to general convex optimization in terms of function values and to nonconvex optimization in terms of gradient norms. Indeed, in the general convex optimization setting, we typically consider Assumptions~\ref{ass:global_lipschitz_constant}, \ref{ass:convex}, and \ref{ass:stochastic_variance_bounded}. The same constructions can be considered with $f(x)-f^*=\frac{\mu}{2}\norm{x-x^*}^2,$ where $\mu = L.$ For sufficiently small $\varepsilon,$ applying the construction in the proof of Theorem~\ref{thm:divergence_3} with squared-distance accuracy $2\varepsilon/\mu$ suggests a lower bound of order $\Omega\big(\left(\frac{1}{n}\sum_{i=1}^n\tau_i\right)\frac{\sigma^2}{\varepsilon n L}\big)$ seconds for finding $\bar{x}$ such that $\Exp{f(\bar{x})}-f^*\leq\varepsilon.$ Similarly, applying the construction with squared-distance accuracy $\varepsilon/L^2$ suggests a lower bound of order $\Omega\big(\left(\frac{1}{n}\sum_{i=1}^n\tau_i\right)\frac{\sigma^2}{\varepsilon n}\big)$ seconds for finding $\bar{x}$ such that $\mathbb E [\norm{\nabla f(\bar{x})}^2]\leq\varepsilon.$ In nonconvex optimization, we consider Assumptions~\ref{ass:global_lipschitz_constant} and \ref{ass:stochastic_variance_bounded}, and may use the same (convex) construction to prove the lower bounds. These observations suggest that the importance of Assumptions~\ref{ass:inter_strong} and~\ref{ass:pl_condition} for improving \emph{the arithmetic-mean dependence} may extend to other optimization settings. However, obtaining tight dependence on other parameters such as $L$ and $\varepsilon$ would require a different construction and is left for future work.
\section{Conclusions}
\label{sec:conclusions}
In this work, we investigated various assumptions and setups in an effort to break the pessimistic dependence on $\{\tau_i\}$ achieved by \algname{Malenia SGD}. We considered the first- and second-order similarity, strong growth, and interpolation assumptions. We proved that under the first- and second-order similarity assumptions, it is infeasible to improve the dependence on the arithmetic mean of $\{\tau_i\}$ with any randomized algorithm. We also showed that under weak interpolation (Assumption~\ref{ass:inter}), it is likewise not possible for any randomized algorithm to improve upon the result of \algname{Malenia SGD}. Subsequently, we presented new theoretical results that provide improved time complexity guarantees in the heterogeneous setting, without assuming that the functions ${f_i}$ are identical (Theorem~\ref{thm:main_theorem}). These results are obtained under the standard assumptions of convex optimization, together with Assumptions~\ref{ass:inter_strong} and~\ref{ass:pl_condition}. 
Importantly, we have not merely introduced these assumptions to close the gap, but have shown that neither Assumption~\ref{ass:inter_strong} nor Assumption~\ref{ass:pl_condition} can be dropped in general within our setting, highlighting the fundamental limits of heterogeneous stochastic optimization. At the same time, we acknowledge that strong interpolation is a restrictive assumption in heterogeneous settings, as it requires the local functions to share the same set of minimizers.
Identifying this fundamental difficulty is one of the main goals of our work.

There are many unexplored directions that can build on our initial results and observations. While we focused on the most common assumptions in federated and distributed learning, our findings may inspire the development of new assumptions and settings where it is possible to improve upon \algname{Malenia SGD}. 
Moreover, our upper bounds and lower bounds were investigated in terms of $\mathbb{E}[\|x^k - x_*\|^2]$ convergence, and the lower bounds are only tight up to logarithmic factors and in small-$\varepsilon$ regimes. 
It would be interesting to see whether tight lower bounds can be obtained in terms of $\mathbb{E}[\|\nabla f(x^k)\|^2]$ in the non-convex setting, and in terms of $\mathbb{E}[f(x^k)] - f^*$ for convex functions.

\bibliography{iclr2027_conference}

@book{goodfellow2016deep,
title={Deep learning},
author={Goodfellow, Ian and Bengio, Yoshua and Courville, Aaron and Bengio, Yoshua},
volume={1},
year={2016},
publisher={MIT Press}
}

@inproceedings{mcmahan2017communication,
  title={Communication-efficient learning of deep networks from decentralized data},
  author={McMahan, Brendan and Moore, Eider and Ramage, Daniel and Hampson, Seth and y Arcas, Blaise Aguera},
  booktitle={Artificial Intelligence and Statistics},
  pages={1273--1282},
  year={2017},
  organization={PMLR}
}

@book{nesterov2018lectures,
	title={Lectures on convex optimization},
	author={Nesterov, Yurii},
	volume={137},
	year={2018},
	publisher={Springer}
}

@inproceedings{he2016deep,
	title={Deep residual learning for image recognition},
	author={He, Kaiming and Zhang, Xiangyu and Ren, Shaoqing and Sun, Jian},
	booktitle={Proceedings of the IEEE Conference on Computer Vision and Pattern Recognition (CVPR)},
	pages={770--778},
	year={2016}
}

@techreport{krizhevsky2009learning,
	title={Learning multiple layers of features from tiny images},
	author={Krizhevsky, Alex and Hinton, Geoffrey and others},
	year={2009},
	jnumber = {Technical Report TR-2009},
	institution = {University of Toronto,  Toronto}
}

@article{konevcny2016federated,
  title={Federated learning: Strategies for improving communication efficiency},
  author={Kone{\v{c}}n{\'y}, Jakub and McMahan, H Brendan and Yu, Felix X and Richt{\'a}rik, Peter and Suresh, Ananda Theertha and Bacon, Dave},
  journal={arXiv preprint arXiv:1610.05492},
  year={2016}
}

@article{carmon2020lower,
  title={Lower bounds for finding stationary points I},
  author={Carmon, Yair and Duchi, John C and Hinder, Oliver and Sidford, Aaron},
  journal={Mathematical Programming},
  volume={184},
  number={1},
  pages={71--120},
  year={2020},
  publisher={Springer}
}

@article{huang2022lower,
  title={Lower Bounds and Nearly Optimal Algorithms in Distributed Learning with Communication Compression},
  author={Huang, Xinmeng and Chen, Yiming and Yin, Wotao and Yuan, Kun},
  journal={Advances in Neural Information Processing Systems (NeurIPS)},
  year={2022}
}

@article{schmidt2013fast,
  title={Fast convergence of stochastic gradient descent under a strong growth condition},
  author={Schmidt, Mark and Roux, Nicolas Le},
  journal={arXiv preprint arXiv:1308.6370},
  year={2013}
}

@article{goyal2017accurate,
  title={Accurate, large minibatch {SGD}: Training imagenet in 1 hour},
  author={Goyal, Priya and Doll{\'a}r, Piotr and Girshick, Ross and Noordhuis, Pieter and Wesolowski, Lukasz and Kyrola, Aapo and Tulloch, Andrew and Jia, Yangqing and He, Kaiming},
  journal={arXiv preprint arXiv:1706.02677},
  year={2017}
}

@article{koloskova2022sharper,
  title={Sharper Convergence Guarantees for Asynchronous {SGD} for Distributed and Federated Learning},
  author={Koloskova, Anastasia and Stich, Sebastian U and Jaggi, Martin},
  journal={Advances in Neural Information Processing Systems (NeurIPS)},
  year={2022}
}

@inproceedings{lu2021optimal,
  title={Optimal complexity in decentralized training},
  author={Lu, Yucheng and De Sa, Christopher},
  booktitle={International Conference on Machine Learning},
  pages={7111--7123},
  year={2021},
  organization={PMLR}
}

@book{lan2020first,
  title={First-order and stochastic optimization methods for machine learning},
  author={Lan, Guanghui},
  year={2020},
  publisher={Springer}
}

@article{arjevani2022lower,
  title={Lower bounds for non-convex stochastic optimization},
  author={Arjevani, Yossi and Carmon, Yair and Duchi, John C and Foster, Dylan J and Srebro, Nathan and Woodworth, Blake},
  journal={Mathematical Programming},
  pages={1--50},
  year={2022},
  publisher={Springer}
}

@article{nemirovskij1983problem,
  title={Problem complexity and method efficiency in optimization},
  author={Nemirovskij, Arkadij Semenovi{\v{c}} and Yudin, David Borisovich},
  year={1983},
  publisher={Wiley-Interscience}
}

@article{cotter2011better,
  title={Better mini-batch algorithms via accelerated gradient methods},
  author={Cotter, Andrew and Shamir, Ohad and Srebro, Nati and Sridharan, Karthik},
  journal={Advances in Neural Information Processing Systems},
  volume={24},
  year={2011}
}

@inproceedings{gower2019sgd,
  title={{SGD}: General analysis and improved rates},
  author={Gower, Robert Mansel and Loizou, Nicolas and Qian, Xun and Sailanbayev, Alibek and Shulgin, Egor and Richt{\'a}rik, Peter},
  booktitle={International Conference on Machine Learning},
  pages={5200--5209},
  year={2019},
  organization={PMLR}
}

@inproceedings{dutta2018slow,
  title={Slow and stale gradients can win the race: Error-runtime trade-offs in distributed {SGD}},
  author={Dutta, Sanghamitra and Joshi, Gauri and Ghosh, Soumyadip and Dube, Parijat and Nagpurkar, Priya},
  booktitle={International Conference on Artificial Intelligence and Statistics},
  pages={803--812},
  year={2018},
  organization={PMLR}
}

@inproceedings{arjevani2020tight,
  title={A tight convergence analysis for stochastic gradient descent with delayed updates},
  author={Arjevani, Yossi and Shamir, Ohad and Srebro, Nathan},
  booktitle={Algorithmic Learning Theory},
  pages={111--132},
  year={2020},
  organization={PMLR}
}

@article{feyzmahdavian2016asynchronous,
  title={An asynchronous mini-batch algorithm for regularized stochastic optimization},
  author={Feyzmahdavian, Hamid Reza and Aytekin, Arda and Johansson, Mikael},
  journal={IEEE Transactions on Automatic Control},
  volume={61},
  number={12},
  pages={3740--3754},
  year={2016},
  publisher={IEEE}
}

@article{mishchenko2022asynchronous,
  title={Asynchronous {SGD} beats minibatch {SGD} under arbitrary delays},
  author={Mishchenko, Konstantin and Bach, Francis and Even, Mathieu and Woodworth, Blake},
  journal={Advances in Neural Information Processing Systems (NeurIPS)},
  year={2022}
}

@article{wu2022delay,
  title={Delay-adaptive step-sizes for asynchronous learning},
  author={Wu, Xuyang and Magnusson, Sindri and Feyzmahdavian, Hamid Reza and Johansson, Mikael},
  journal={arXiv preprint arXiv:2202.08550},
  year={2022}
}

@article{woodworth2018graph,
  title={Graph oracle models, lower bounds, and gaps for parallel stochastic optimization},
  author={Woodworth, Blake E and Wang, Jialei and Smith, Adam and McMahan, Brendan and Srebro, Nati},
  journal={Advances in Neural Information Processing Systems},
  volume={31},
  year={2018}
}

@article{cohen2021asynchronous,
  title={Asynchronous Stochastic Optimization Robust to Arbitrary Delays},
  author={Cohen, Alon and Daniely, Amit and Drori, Yoel and Koren, Tomer and Schain, Mariano},
  journal={Advances in Neural Information Processing Systems},
  volume={34},
  pages={9024--9035},
  year={2021}
}

@article{tyurin2023optimal,
  title={Optimal Time Complexities of Parallel Stochastic Optimization Methods Under a Fixed Computation Model},
  author={Tyurin, Alexander and Richt{\'a}rik, Peter},
  journal = {Advances in Neural Information Processing Systems (NeurIPS)},
  year = {2023},
}

@inproceedings{szlendak2021permutation,
  title={Permutation Compressors for Provably Faster Distributed Nonconvex Optimization},
  author={Szlendak, Rafa{\l} and Tyurin, Alexander and Richt{\'a}rik, Peter},
  booktitle={International Conference on Learning Representations},
  year={2021}
}

@inproceedings{scaman2017optimal,
  title={Optimal algorithms for smooth and strongly convex distributed optimization in networks},
  author={Scaman, Kevin and Bach, Francis and Bubeck, S{\'e}bastien and Lee, Yin Tat and Massouli{\'e}, Laurent},
  booktitle={International Conference on Machine Learning},
  pages={3027--3036},
  year={2017},
  organization={PMLR}
}

@article{woodworth2016tight,
  title={Tight complexity bounds for optimizing composite objectives},
  author={Woodworth, Blake E and Srebro, Nati},
  journal={Advances in Neural Information Processing Systems},
  volume={29},
  year={2016}
}

@article{arjevani2015communication,
  title={Communication complexity of distributed convex learning and optimization},
  author={Arjevani, Yossi and Shamir, Ohad},
  journal={Advances in Neural Information Processing Systems},
  volume={28},
  year={2015}
}

@inproceedings{karimi2016linear,
  title={Linear convergence of gradient and proximal-gradient methods under the polyak-{\l}ojasiewicz condition},
  author={Karimi, Hamed and Nutini, Julie and Schmidt, Mark},
  booktitle={Machine Learning and Knowledge Discovery in Databases: European Conference, ECML PKDD 2016, Riva del Garda, Italy, September 19-23, 2016, Proceedings, Part I 16},
  pages={795--811},
  year={2016},
  organization={Springer}
}

@article{feyzmahdavian2023asynchronous,
  title={Asynchronous iterations in optimization: New sequence results and sharper algorithmic guarantees},
  author={Feyzmahdavian, Hamid Reza and Johansson, Mikael},
  journal={Journal of Machine Learning Research},
  volume={24},
  number={158},
  pages={1--75},
  year={2023}
}

@inproceedings{tyurin2024tight,
  title={Tight Time Complexities in Parallel Stochastic Optimization with Arbitrary Computation Dynamics}, 
  author={Tyurin, Alexander},
  booktitle={International Conference on Learning Representations (ICLR)},
  year={2025},
}

@article{zou2019improved,
  title={An improved analysis of training over-parameterized deep neural networks},
  author={Zou, Difan and Gu, Quanquan},
  journal={Advances in Neural Information Processing Systems},
  volume={32},
  year={2019}
}

@article{zhang2021understanding,
  title={Understanding deep learning (still) requires rethinking generalization},
  author={Zhang, Chiyuan and Bengio, Samy and Hardt, Moritz and Recht, Benjamin and Vinyals, Oriol},
  journal={Communications of the ACM},
  volume={64},
  number={3},
  pages={107--115},
  year={2021},
  publisher={ACM New York, NY, USA}
}

@article{tyurin2024freya,
  title={Freya {PAGE}: First Optimal Time Complexity for Large-Scale Nonconvex Finite-Sum Optimization with Heterogeneous Asynchronous Computations},
  author={Tyurin, Alexander and Gruntkowska, Kaja and Richt{\'a}rik, Peter},
  journal = {Advances in Neural Information Processing Systems (NeurIPS)},
  year = {2024},
}

@article{vaswani2019painless,
  title={Painless stochastic gradient: Interpolation, line-search, and convergence rates},
  author={Vaswani, Sharan and Mishkin, Aaron and Laradji, Issam and Schmidt, Mark and Gidel, Gauthier and Lacoste-Julien, Simon},
  journal={Advances in neural information processing systems},
  volume={32},
  year={2019}
}

@inproceedings{maranjyan2025ringmaster,
  title={Ringmaster {ASGD}: The First Asynchronous {SGD} with Optimal Time Complexity},
  author={Maranjyan, Artavazd and Tyurin, Alexander and Richt{\'a}rik, Peter},
  booktitle={International Conference on Machine Learning},
  year={2025}
}

@article{lian2015asynchronous,
  title={Asynchronous parallel stochastic gradient for nonconvex optimization},
  author={Lian, Xiangru and Huang, Yijun and Li, Yuncheng and Liu, Ji},
  journal={Advances in Neural Information Processing Systems},
  volume={28},
  year={2015}
}

@article{stich2020error,
  title={The error-feedback framework: {SGD} with delayed gradients},
  author={Stich, Sebastian U and Karimireddy, Sai Praneeth},
  journal={Journal of Machine Learning Research},
  volume={21},
  number={237},
  pages={1--36},
  year={2020}
}

@inproceedings{sra2016adadelay,
  title={Adadelay: Delay adaptive distributed stochastic optimization},
  author={Sra, Suvrit and Yu, Adams Wei and Li, Mu and Smola, Alex},
  booktitle={Artificial Intelligence and Statistics},
  pages={957--965},
  year={2016},
  organization={PMLR}
}

@inproceedings{islamov2024asgrad,
  title={{AsGrad}: A sharp unified analysis of asynchronous-{SGD} algorithms},
  author={Islamov, Rustem and Safaryan, Mher and Alistarh, Dan},
  booktitle={International Conference on Artificial Intelligence and Statistics},
  pages={649--657},
  year={2024},
  organization={PMLR}
}

@article{zhang2015staleness,
  title={Staleness-aware async-sgd for distributed deep learning},
  author={Zhang, Wei and Gupta, Suyog and Lian, Xiangru and Liu, Ji},
  journal={arXiv preprint arXiv:1511.05950},
  year={2015}
}

@book{wainwright2019high,
  title={High-dimensional statistics: A non-asymptotic viewpoint},
  author={Wainwright, Martin J},
  volume={48},
  year={2019},
  publisher={Cambridge university press}
}

@inproceedings{maranjyan2026ringleader,
  title={Ringleader ASGD: The first asynchronous SGD with optimal time complexity under data heterogeneity},
  author={Maranjyan, Artavazd and Richt{\'a}rik, Peter},
  booktitle={International Conference on Learning Representations},
  volume={2026},
  pages={4738--4768},
  year={2026}
}
\bibliographystyle{iclr2027_conference}

\appendix

\newpage

\tableofcontents

\newpage

\newpage

\section{Arbitrarily computation dynamics}
\label{sec:arbitrarily}
Our new result can be readily extended to \emph{the universal computation model}. To encompass virtually all computation scenarios, assume that each worker $i$ performs computations based on a \emph{computation power} function $v_i \,:\, \R_{+} \rightarrow \R_{+}$. Then the number of stochastic gradients that worker $i$ can calculate from a time $t_0$ to a time $t_1$ is an integral of the computation power $v_i$ followed by the floor operation:
\begin{align}
  \label{eq:required_time}
  \textnormal{``\# of stoch. grad. in $[t_0, t_1]$''} = \flr{\int_{t_0}^{t_1} v_i(\tau) d \tau}.
\end{align}
For instance, if worker $i$ is inactive for the first $t$ seconds and then active again, it would mean $v_i(\tau) = 0$ for all $\tau \leq t$ and $v_i(\tau) > 0$ for all $\tau > t.$ Using the universal computation model, we can prove the theorem:
\begin{restatable}{theorem}{MAINTHEOREMARBITRARILY}
  Consider the assumptions, algorithm, and parameters from Theorem~\ref{thm:main_theorem}. Then, \algname{Rennala SGD} converges after at most $\bar{t}_{\ceil{c \times \frac{L_{\max}}{\mu} \log \frac{R^2}{\varepsilon}}}$ seconds, where the sequence $\{\bar{t}_k\}$ is defined recursively as $\bar{t}_k \eqdef $
  \begin{align}
    \label{eq:homog_compl}
    \squeeze \min\left\{t \geq 0 : \sum\limits_{i=1}^n \flr{{\displaystyle \int_{\bar{t}_{k-1}}^{t}} v_i(\tau) d \tau} \geq \max\left\{\ceil{2S}, 1\right\}\right\}
  \end{align}
  for all $k \geq 1$ $(\bar{t}_0 \equiv 0),$ and $c$ is a universal constant.
\end{restatable}
A similar result was obtained in \citep{tyurin2024tight}. However, \citet{tyurin2024tight} requires the equality of the functions $\{f_i\}.$

\section{Proof of the Main Results}
\label{sec:proofs}

\MAINTHEOREM*

\begin{proof}
  Let us define $x_*^{k}$ as an euclidean projection of the point $x^{k+1}$ on to the solution set of the main problem \eqref{eq:main_task}, and take the condition expectation $\ExpSub{k}{\cdot}$ w.r.t. the randomness from the iteration $k$ only.
  Then we have 
  \begin{align*}
    \ExpSub{k}{\norm{x^{k+1} - x_*^{k+1}}^2} \leq \ExpSub{k}{\norm{x^{k+1} - x_*^{k}}^2}
  \end{align*}
  due to the projection's properties. Then
  \begin{align*}
    &\ExpSub{k}{\norm{x^{k+1} - x_*^{k+1}}^2} \\
    &\leq \ExpSub{k}{\norm{x^{k} - \gamma \frac{1}{n} \sum_{i=1}^{n} w_i^k \sum_{j=1}^{B_i^k} \nabla f_i(x^k;\xi^k_{ij}) - x_*^{k}}^2} \\
    &= \norm{x^{k} - x_*^{k}}^2 - 2 \gamma \ExpSub{k}{\inp{x^{k} - x_*^{k}}{\frac{1}{n} \sum_{i=1}^{n} w_i^k \sum_{j=1}^{B_i^k} \nabla f_i(x^k;\xi^k_{ij})}} + \gamma^2 \ExpSub{k}{\norm{\frac{1}{n} \sum_{i=1}^{n} w_i^k \sum_{j=1}^{B_i^k} \nabla f_i(x^k;\xi^k_{ij})}^2}.
  \end{align*}
  Using unbiasedness (Assumption~\ref{ass:stochastic_variance_bounded}) and the variance decomposition equality, we get
  \begin{align*}
    &\ExpSub{k}{\norm{x^{k+1} - x_*^{k+1}}^2} \\
    &\leq \norm{x^{k} - x_*^{k}}^2 - 2 \gamma \inp{x^{k} - x_*^{k}}{\frac{1}{n} \sum_{i=1}^{n} w_i^k B_i^k \nabla f_i(x^k)} \\
    &\quad + \gamma^2 \norm{\frac{1}{n} \sum_{i=1}^{n} w_i^k B_i^k \nabla f_i(x^k)}^2 + \gamma^2 \ExpSub{k}{\norm{\frac{1}{n} \sum_{i=1}^{n} w_i^k \sum_{j=1}^{B_i^k} \left(\nabla f_i(x^k;\xi^k_{ij}) - \nabla f_i(x^k)\right)}^2}.
  \end{align*}
  Consider the last term, due to the independence of stochastic gradients and Assumption~\ref{ass:stochastic_variance_bounded}, we ensure that
  \begin{align*}
    &\ExpSub{k}{\norm{\frac{1}{n} \sum_{i=1}^{n} w_i^k \sum_{j=1}^{B_i^k} \left(\nabla f_i(x^k;\xi^k_{ij}) - \nabla f_i(x^k)\right)}^2} \\
    &=\frac{1}{n^2} \sum_{i=1}^{n} (w_i^k)^2 \sum_{j=1}^{B_i^k} \ExpSub{k}{\norm{\nabla f_i(x^k;\xi^k_{ij}) - \nabla f_i(x^k)}^2} \leq \frac{1}{n^2} \sum_{i=1}^{n} (w_i^k)^2 B_i^k \sigma^2.
  \end{align*}
  Thus
  \begin{align}
    &\ExpSub{k}{\norm{x^{k+1} - x_*^{k+1}}^2} \label{eq:iHCtnSUNn}\\
    &\leq \norm{x^{k} - x_*^{k}}^2 - \frac{2 \gamma}{n} \sum_{i=1}^{n} w_i^k B_i^k \inp{x^{k} - x_*^{k}}{\nabla f_i(x^k)} + \gamma^2 \norm{\frac{1}{n} \sum_{i=1}^{n} w_i^k B_i^k \nabla f_i(x^k)}^2 + \frac{\gamma^2}{n^2} \sum_{i=1}^{n} (w_i^k)^2 B_i^k \sigma^2. \nonumber
  \end{align}
  We now consider the second and the third term. Since $\frac{1}{n} \sum_{i=1}^{n} w_i^k B_i^k = 1,$ using Jensen's inequality, we get
  \begin{align*}
    &- 2 \gamma \inp{x^{k} - x_*^{k}}{\frac{1}{n} \sum_{i=1}^{n} w_i^k B_i^k \nabla f_i(x^k)} + \gamma^2 \norm{\frac{1}{n} \sum_{i=1}^{n} w_i^k B_i^k \nabla f_i(x^k)}^2 \\
    &\leq - 2 \gamma \inp{x^{k} - x_*^{k}}{\frac{1}{n} \sum_{i=1}^{n} w_i^k B_i^k \nabla f_i(x^k)} + \gamma^2 \frac{1}{n} \sum_{i=1}^{n} w_i^k B_i^k \norm{\nabla f_i(x^k)}^2.
  \end{align*}
  Due to Assumption~\ref{ass:inter_strong}, we get
  \begin{align*}
    &\frac{1}{n} \sum_{i=1}^{n} w_i^k B_i^k \norm{\nabla f_i(x^k)}^2 = \frac{1}{n} \sum_{i=1}^{n} w_i^k B_i^k \norm{\nabla f_i(x^k) - \nabla f_i(x_*^{k})}^2 \\
    &\leq \frac{1}{n} \sum_{i=1}^{n} w_i^k B_i^k L_i \inp{x^k - x_*^{k}}{\nabla f_i(x^k) - \nabla f_i(x_*^{k})} \\
    &\leq L_{\max} \frac{1}{n} \sum_{i=1}^{n} w_i^k B_i^k \inp{x^k - x_*^{k}}{\nabla f_i(x^k) - \nabla f_i(x_*^{k})} \\
    &= L_{\max} \frac{1}{n} \sum_{i=1}^{n} w_i^k B_i^k \inp{x^k - x_*^{k}}{\nabla f_i(x^k)}.
  \end{align*}
  In the first inequality, we use Lemma~\ref{lemma:lipt_func} under Assumption~\ref{ass:lipschitz_constant} and convexity (Assumption~\ref{ass:convex}). In the second inequality, we use the bound $L_i \leq L_{\max}$ for all $i \in [n].$ Taking $\gamma \leq \nicefrac{1}{L_{\max}}$ and substituting the last inequality to \eqref{eq:iHCtnSUNn}, we obtain
  \begin{align*}
    &\ExpSub{k}{\norm{x^{k+1} - x_*^{k+1}}^2} \\
    &\leq \norm{x^{k} - x_*^{k}}^2 - (2 \gamma - L_{\max} \gamma^2)\frac{1}{n} \sum_{i=1}^{n} w_i^k B_i^k \inp{x^{k} - x_*^{k}}{\nabla f_i(x^k)} + \frac
    {\gamma^2}{n^2} \sum_{i=1}^{n} (w_i^k)^2 B_i^k \sigma^2 \\
    &\leq \norm{x^{k} - x_*^{k}}^2 - \gamma \frac{1}{n} \sum_{i=1}^{n} w_i^k B_i^k \inp{x^{k} - x_*^{k}}{\nabla f_i(x^k)} + \frac
    {\gamma^2}{n^2} \sum_{i=1}^{n} (w_i^k)^2 B_i^k \sigma^2.
  \end{align*}
  Using the convexity, Assumption~\ref{ass:pl_condition}, and Lemma~\ref{lemma:star_strongly_convex}, we get
  \begin{align*}
    &\ExpSub{k}{\norm{x^{k+1} - x_*^{k+1}}^2} \\
    &\leq \norm{x^{k} - x_*^{k}}^2 - \frac{\gamma \mu}{2} \frac{1}{n} \sum_{i=1}^{n} w_i^k B_i^k \norm{x^{k} - x_*^{k}}^2 + \frac
    {\gamma^2}{n^2} \sum_{i=1}^{n} (w_i^k)^2 B_i^k \sigma^2.
  \end{align*}
  We take $w_i^k = \nicefrac{n}{\sum_{i=1}^n B_i^k}$ in the theorem for all $i \in [n].$ Thus
  \begin{align*}
    \ExpSub{k}{\norm{x^{k+1} - x_*^{k+1}}^2} \leq \left(1 - \frac{\gamma \mu}{2}\right)\norm{x^{k} - x_*^{k}}^2 + \frac
    {\gamma^2 \sigma^2}{\sum_{i=1}^n B_i^k}.
  \end{align*}
  In Algorithm~\ref{alg:alg_server_heterog}, with the chosen weights $\{w_i^k\}$, we wait for the moment when $\sum_{i=1}^n B_i^k > S.$ Thus
  \begin{align*}
    \ExpSub{k}{\norm{x^{k+1} - x_*^{k+1}}^2} \leq \left(1 - \frac{\gamma \mu}{2}\right)\norm{x^{k} - x_*^{k}}^2 + \frac
    {\gamma^2 \sigma^2}{S}.
  \end{align*}
  Unrolling the recursion and taking the full expectation, we obtain
  \begin{align*}
    \Exp{\norm{x^{k+1} - x_*^{k+1}}^2} 
    &\leq \left(1 - \frac{\gamma \mu}{2}\right)^{k + 1}\norm{x^{0} - x_*^{0}}^2 + \sum_{j=0}^{k} \left(1 - \frac{\gamma \mu}{2}\right)^{j} \frac{\gamma^2 \sigma^2}{S} \\
    &\leq \left(1 - \frac{\gamma \mu}{2}\right)^{k + 1}\norm{x^{0} - x_*^{0}}^2 + \frac{2 \gamma \sigma^2}{\mu S}.
  \end{align*}
  Due the choice of $\gamma,$ $S,$ and the condition on $k,$ we have $\Exp{\norm{x^{k+1} - x_*^{k+1}}^2} \leq \varepsilon.$

  It is sufficient to run the method for
  \begin{align*}
    \cO\left(\frac{L_{\max}}{\mu} \log \frac{R^2}{\varepsilon}\right)
  \end{align*}
  iterations. In each iteration, the method has to ensure that $\sum_{i=1}^n B_i^k > S.$ A sufficient time for that is 
  \begin{align*}
    2 \min\limits_{m \in [n]} \left[\left(\frac{1}{m} \sum\limits_{i=1}^{m} \frac{1}{\tau_{i}}\right)^{-1} \left(1 + \frac{4 \sigma^2}{m L_{\max} \varepsilon \mu}\right)\right].
  \end{align*}
  under the fixed computation model (see Theorem~11 in \citep{tyurin2024freya}).
\end{proof}

\MAINTHEOREMARBITRARILY*

\begin{proof}
  From the proof of Theorem~\ref{thm:main_theorem}, we know that it is sufficient to run the method for
  \begin{align*}
    c \times \frac{L_{\max}}{\mu} \log \frac{R^2}{\varepsilon}
  \end{align*}
  iterations, where $c$ is a universal constant. The method waits the moment when $\sum_{i=1}^n B_i^k > S$ in each iteration. The workers work in parallel, and for all $i \in [n],$ will calculate 
  \begin{align*}
    \flr{\int_{0}^{t} v_i(\tau) d \tau}
  \end{align*}
  stochastic gradients after $t$ seconds. In total, all workers will calculate $\sum_{i=1}^n \flr{\int_{0}^{t} v_i(\tau) d \tau}$ stochastic gradients. Hence, the first iteration will end by
  \begin{align*}
    \bar{t}_1 \eqdef \min\left\{t \geq 0 \, : \, \sum_{i=1}^n \flr{\int_{0}^{t} v_i(\tau) d \tau} \geq \max\left\{\ceil{2S},1\right\}\right\},
  \end{align*}
  seconds. 
  After that, the second iteration starts before time $\bar{t}_1$ and ends no later than time
  \begin{align*}
    \bar{t}_2 \eqdef \min\left\{t \geq 0 \, : \, \sum_{i=1}^n \flr{\int_{\bar{t}_1}^{t} v_i(\tau) d \tau} \geq \max\left\{\ceil{2S},1\right\}\right\},
  \end{align*}
  because worker $i$ can calculate at least 
  \begin{align*}
    \flr{\int_{\bar{t}_1}^{t} v_i(\tau) d \tau}
  \end{align*}
  stochastic gradients between the end of the first iteration and a time $t.$ Using the same reasoning, we can recursively define 
  $$\bar{t}_3, \dots, \bar{t}_{\ceil{c \times \frac{L_{\max}}{\mu} \log \frac{R^2}{\varepsilon}}}.$$ The algorithm will converge by $\bar{t}_{\ceil{c \times \frac{L_{\max}}{\mu} \log \frac{R^2}{\varepsilon}}}$ seconds due to the discussion at the beginning of the theorem.
\end{proof}

\section{Proof of Lower Bounds}

\FIRSTORDER*

\begin{proof}
We assume that $\varphi_i \in [-1, 1]$ for all $i \in [n]$ and define them later. The first-order similarity of these functions is 
\begin{align*}
    \max_{x \in \R^n} \max_{i,j \in [n]} \norm{\nabla f_i(x) - \nabla f_j(x)}^2 \leq 2 \beta^2.
\end{align*}
Thus, the parameter $\beta$ from the construction controls this similarity. Taking $\beta$ small, we increase similarity between the functions. Notice that the second-order similarity between the functions is zero since $\nabla^2 f_i(x) = \mu \mI$ for all $i \in [n].$ The optimal point is $x^*$ such that $x^*_i = \frac{\beta \varphi_i}{\mu n}$ and $\norm{x^0 - x^*}^2 \leq \frac{\beta^2}{\mu^2 n} \leq R^2.$

Let $N_i(t) \eqdef \left\lfloor\frac{t}{\tau_i}\right\rfloor$ denote an upper bound on the number of stochastic gradients returned by
worker $i$ by time $t$.
For every stochastic gradient returned by worker $i$, the method gets
\begin{align}
  \nabla f_i(\bar{x};\xi) = \mu \bar{x} - \beta \varphi_i e_i + \xi e_i = \mu \bar{x} - \beta \left(\varphi_i - \frac{\xi}{\beta}\right) e_i
\end{align}
where $\bar{x}$ is a query point and $\xi \sim \mathcal N(0,\sigma^2).$
Therefore, all the information obtained from worker $i$ about
$\varphi_i$ consists of at most $N_i(t)$ independent observations
from $\mathcal N(\varphi_i,\frac{\sigma^2}{\beta^2})$.

Let $\bar{x}_t$ be the point returned by the algorithm at time $t$ for the objective parametrized by $\varphi.$ Then,
\begin{align*}
  \norm{\bar{x}_t - x^*}^2 = \sum_{i=1}^n \left((\bar{x}_t)_i - \frac{\beta \varphi_i}{\mu n}\right)^2 = \frac{\beta^2}{\mu^2 n^2} \sum_{i=1}^n \left(\frac{\mu n}{\beta}(\bar{x}_t)_i - \varphi_i\right)^2.
\end{align*}
Thus, we reduced the problem to estimating $\varphi_i$ using the observations $\mathcal N(\varphi_i,\frac{\sigma^2}{\beta^2}).$ Using the classical statistical result (e.g., Example 15.4 from \citep{wainwright2019high}), for a universal constant $c_2>0,$
\begin{align*}
  \sup_{\varphi \in [-1, 1]^n} \Exp{\norm{\bar{x}_t - x^*}^2} \geq c_2 \frac{\beta^2}{\mu^2 n^2} \sum_{i=1}^n \min\left\{1, \frac{\sigma^2}{\beta^2 N_i(t)}\right\},
\end{align*}
where $\bar{x}_t$ depends on $\varphi.$ Thus,
\begin{align*}
  \sup_{\varphi \in [-1, 1]^n} \Exp{\norm{\bar{x}_t - x^*}^2} \geq c_2 \frac{\beta^2}{\mu^2 n^2} \sum_{i=1}^n \min\left\{1, \frac{\sigma^2\tau_i}{\beta^2 t}\right\}.
\end{align*}
We take
\begin{align*}
  T=c_1 \frac{\sigma^2}{\mu^2n^2\varepsilon}\sum_{i=1}^n\tau_i,
\end{align*}
where $c_1$ is a universal constant. 
Then,
\begin{align*}
  \sup_{\varphi \in [-1, 1]^n} \Exp{\norm{\bar{x}_t - x^*}^2} \geq c_2 \frac{\beta^2}{\mu^2 n^2} \sum_{i=1}^n \min\left\{1, \frac{\sigma^2\tau_i}{\beta^2 T}\right\}
\end{align*}
due to the bound on $t$. The condition on $\varepsilon$ ensures that $\frac{\sigma^2\tau_i}{\beta^2 T}\leq 1$ for all $i \in [n].$ Therefore, 
\begin{align}
\label{eq:TSJqq}
\sup_{\varphi \in [-1, 1]^n} \Exp{\norm{\bar{x}_t - x^*}^2} \geq 2\varepsilon.
\end{align}
Finally, $\Exp{\norm{\bar{x}_t - x^*}^2} > \varepsilon$ after $t$ seconds for some $\varphi$ due to \eqref{eq:TSJqq}. The adversary can choose this $\varphi.$
\end{proof}

\WEAKNEEDED*

\begin{proof}
  Without loss of generality, assume that the method starts at $0.$
  Let \(S_\tau = \sum_{i=1}^n \tau_i\) and \(a_i = R\sqrt{\tau_i/S_\tau}\). For all \(i \in [n]\), \(\varphi_i \in [-a_i,a_i]\) is defined later, and
  \begin{align*}
    f_i(x) = \frac{n\mu}{2}\inp{x-\varphi}{e_i}^2,
    \qquad
    \nabla f_i(x;\xi_i) = n\mu\inp{x-\varphi}{e_i}e_i + \xi_i e_i,
  \end{align*}
  where \(\xi_i \sim \mathcal N(0,\sigma^2)\). The average function is
  \begin{align*}
    f(x) = \frac{\mu}{2}\norm{x-\varphi}^2,
  \end{align*}
  and its unique minimizer is \(x^*=\varphi\). Moreover, \(\norm{\varphi}\leq R\).

  Each \(f_i\) is convex and smooth with parameter $n\mu \leq L_{\max}$, and
  \begin{align*}
    \norm{\nabla f_i(x)}^2 = 2n\mu f_i(x) \geq 2\mu f_i(x).
  \end{align*}
  Hence Assumptions~\ref{ass:convex}, \ref{ass:lipschitz_constant}, and \ref{ass:pl_condition} hold. The function \(f\) is smooth with parameter $\mu \leq L_{\max}$ and satisfies the global P\L\,condition with constant \(\mu\). Assumption~\ref{ass:inter} holds because \(\varphi\) minimizes every \(f_i\), while Assumption~\ref{ass:inter_strong} does not hold in general because
  \begin{align*}
    \operatorname*{arg\,min} f_i = \{x \in \R^n : x_i=\varphi_i\} \neq \{\varphi\} = \operatorname*{arg\,min} f.
  \end{align*}
  Finally, the stochastic gradients are unbiased and have variance \(\sigma^2\).

  The rest of the proof is the same as the proof of Theorem~\ref{thm:divergence_first}. Every stochastic gradient returned by worker $i$ gives an observation from $\mathcal N(\varphi_i,\sigma^2/(n^2\mu^2)).$ Let $N_i(t) \eqdef \lfloor t/\tau_i\rfloor$ be an upper bound on the number of stochastic gradients calculated by worker $i$ by time $t.$ The same Gaussian mean-estimation lower bound gives
  \begin{align*}
    \sup_{\substack{\varphi_i\in[-a_i,a_i]\\i\in[n]}}
    \Exp{\norm{\bar{x}_t-\varphi}^2}
    &\geq c_1\sum_{i=1}^n
    \min\left\{a_i^2,\frac{\sigma^2}{n^2\mu^2N_i(t)}\right\} \\
    &\geq c_1\min\left\{R^2,\frac{\sigma^2S_\tau}{n^2\mu^2t}\right\} \geq 2 \varepsilon
  \end{align*}
  for a universal constant $c_1>0,$ where $\bar{x}_t$ any possible query point by time $t$ for the objective with $\varphi.$ The second inequality follows from $N_i(t)\leq t/\tau_i$ and $a_i^2=R^2\tau_i/S_\tau,$ which imply
  \begin{align*}
    \min\left\{a_i^2,\frac{\sigma^2}{n^2\mu^2N_i(t)}\right\}
    \geq \frac{\tau_i}{S_\tau}\min\left\{R^2,\frac{\sigma^2S_\tau}{n^2\mu^2t}\right\},
  \end{align*}
  where remains to sum over $i$ and use $\sum_{i=1}^n\tau_i/S_\tau=1.$
  The third inequality follows from $\varepsilon<0.01,$ $R>10,$ and the condition on $t.$ Finally,
  \begin{align*}
    \Exp{\norm{\bar{x}_t-\varphi}^2} > \varepsilon
  \end{align*}
  for some $\varphi_i\in[-a_i,a_i],$ which the adversary can choose.
\end{proof}

  \ALLNEEDED*

  \begin{proof}
    The theorem is a simple corollary of Theorem~\ref{thm:divergence_2} because Theorem~\ref{thm:divergence_2} is stated under more strict assumptions on the class of the functions and stochastic gradients. The result of Theorem~\ref{thm:divergence_2} holds even under additional Assumption~\ref{ass:pl_condition}.
  \end{proof}

  \PLNEEDED*

  \begin{proof}
    Without loss of generalization, assume that the method starts at $0.$ Let $S_\tau=\sum_{i=1}^n\tau_i,$ take
    \begin{align*}
      a_i=\frac{n\mu\tau_i}{S_\tau},
      \qquad
      f_i(x)=\frac{a_i}{2}(x-\varphi)^2,
      \qquad
      \nabla f_i(x;\xi_i)=a_i(x-\varphi)+\xi_i,
    \end{align*}
    where $\xi_i\sim\mathcal N(0,\sigma^2)$ and $\varphi \in [-R,R]$ is defined later. Since $\frac1n\sum_{i=1}^na_i=\mu,$
    \begin{align*}
      f(x)=\frac{\mu}{2}(x-\varphi)^2.
    \end{align*}
    Thus, $f$ satisfies the global P\L\,condition with constant $\mu,$ and every $f_i$ is convex and $a_i$--smooth with $a_i\leq n\mu\leq L_{\max}.$ Moreover, all the functions have the same unique minimizer $\varphi,$ so Assumptions~\ref{ass:inter} and \ref{ass:inter_strong} hold. Notice that $\norm{\nabla f_i(x)}^2=2a_i(f_i(x)-f_i^*).$ 
    Finally, the stochastic gradients are unbiased and have variance $\sigma^2.$

    Every stochastic gradient returned by worker $i$ gives an observation from $\mathcal N(\varphi,\sigma^2/a_i^2).$ Let $N_i(t) \eqdef \lfloor t/\tau_i\rfloor$ be an upper bound on the number of stochastic gradients computed by worker $i$ by time $t.$ Let $P_\varphi^t$ be the joint distribution of these observations by time $t.$ 
    For Gaussians with the same variance, the KL divergence $D_{\mathrm{KL}}(\mathcal N(m_1,v)\,\|\,\mathcal N(m_0,v))=(m_1-m_0)^2/(2v).$ 
  By the chain rule for KL divergence,
  \begin{align*}
  D_{\mathrm{KL}}\!\left(P_\varphi^t\,\middle\|\,P_{-\varphi}^t\right)
  &\leq \sum_{i=1}^{n} \frac{2 \varphi^2 a_i^2}{\sigma^2} \times \frac{t}{\tau_i} = \frac{2 t \varphi^2 n^2 \mu^2}{\sigma^2 S_{\tau}}
  \end{align*}
  because $N_i(t) \leq \frac{t}{\tau_i}$.
  Since
    $t \leq \frac{1}{100} \left(\frac1n\sum_{i=1}^n\tau_i\right) \frac{\sigma^2}{\varepsilon n\mu^2},$
  we get
  \begin{align*}
    D_{\mathrm{KL}}\!\left(P_\varphi^t\,\middle\|\,P_{-\varphi}^t\right) \leq \frac{\varphi^2}{32 \varepsilon}.
  \end{align*}
  Choosing any
  \begin{align*}
    \varphi \in \{-4 \sqrt{\varepsilon}, 4 \sqrt{\varepsilon}\},
  \end{align*} we ensure that $-\varphi,\varphi\in[-R,R]$ (since $\varepsilon<0.01$ and $R>10$) and $D_{\mathrm{KL}}\!\left(P_\varphi^t\,\middle\|\,P_{-\varphi}^t\right) \leq \frac{1}{2}.$
  Pinsker's inequality gives $\norm{P_\varphi^t-P_{-\varphi}^t}_{\mathrm{TV}}\leq 1/2.$ Hence, the two-point Le Cam bound \citep{wainwright2019high} gives
    \begin{align*}
      \max_{\varphi \in \{-4 \sqrt{\varepsilon},4 \sqrt{\varepsilon}\}}
      \Exp{\abs{\bar{x}_t - \varphi}^2}
      \geq 4 \varepsilon,
    \end{align*}
    where $\bar{x}_t$ is the random variable that the random algorithm can produce based on all available information up to time $t$ for the objective with parameter $\varphi.$ Thus, the adversary can take $\varphi = \arg\max_{\varphi \in \{-4 \sqrt{\varepsilon},4 \sqrt{\varepsilon}\}} \Exp{\abs{\bar{x}_t - \varphi}^2}$ to get $\Exp{\abs{\bar{x}_t - \varphi}^2} > \varepsilon.$
  \end{proof}

\section{Auxiliary Results}
In this section, we present well-known results from optimization.

\begin{lemma}[\cite{nesterov2018lectures}]
  \label{lemma:lipt_func}
  Let $f:\R^d\to \R$ be a function, which $L$--smooth and convex. Then for all $x, y\in\R^d$ we have:
  \begin{align}
    \norm{\nabla f(x) - \nabla f(y)}^2 \leq L \inp{\nabla f(x) - \nabla f(y)}{x - y}.
\end{align}
\end{lemma}

\begin{lemma}[\cite{karimi2016linear}]
  \label{lemma:star_strongly_convex}
  Let $f:\R^d\to \R$ be a convex function, which satisfies P\L\,condition with a parameter $\mu$ (Assumption~\ref{ass:pl_global}). Then, for all $x \in \R^d,$ we have
  \begin{align}
      \inp{\nabla f(x)}{x - \bar{x}_*} \geq \frac{\mu}{2} \norm{\bar{x}_* - x}^2
  \end{align}
  where $\bar{x}_*$ is the projection of $x$ onto the solution set of $\min\limits_{x \in \R^d} f(x).$
\end{lemma}

\section{Lower Bound in the Heterogeneous Convex Setting}

This section complements the results from \citep{tyurin2023optimal}, where the authors only prove the optimal time complexities in the \emph{homogeneous nonconvex}, \emph{heterogeneous nonconvex}, and \emph{homogeneous convex} settings. Here, we resolve the last piece, the \emph{heterogeneous convex} setting. Following \citet{tyurin2023optimal}, we have to formalize and introduce the following protocol and classes.

We investigate the optimization problem \eqref{eq:main_task} when the function $f$ is convex. For the convex case, using Protocol~\ref{alg:time_multiple_oracle_protocol}, we use the complexity measure
\begin{align}\label{eq:lower_compl_time_convex}
    \begin{split}
    \squeeze
    &\mathfrak{m}_{\textnormal{time}}\left(\cA, \cF\right) \eqdef \inf_{A \in \cA} \inf\left\{t \geq 0\,\middle|\, \sup_{f \in \cF} \sup_{(O, \mathcal{D}) \in \cO(f)} \left(\Exp{f(x^{k(t)})} - \inf_{x \in Q} f(x)\right) \leq \varepsilon \right\},
\end{split}
\end{align}
where $x^k$ is generated by Protocol~\ref{alg:time_multiple_oracle_protocol}, $k(t)$ is the largest index such that $t^{k(t)} \leq t,$ and $Q$ is a convex set. Let us take any set $Q,$ and consider the following class of convex functions. 
\begin{protocol}[t]
  \caption{Time Multiple Oracles Protocol}
  \label{alg:time_multiple_oracle_protocol}
  \begin{algorithmic}[1]
  \STATE \textbf{Input: }function(s) $f \in \cF,$ oracles and distributions $((O_1, ..., O_n), (\mathcal{D}_1, ..., \mathcal{D}_n)) \in \cO(f),$ algorithm $A~\in~\cA$
  \STATE $s^0_i = 0$ for all $i \in [n]$
  \FOR{$k = 0, \dots, \infty$}
\STATE $({t^{k+1}}, i^{k+1}, x^k) = A^k(g^1, \dots, g^{k}),$ \hfill $\rhd \,{t^{k+1} \geq t^{k}}$
  \STATE $(s^{k+1}_{{i^{k+1}}}, g^{k+1}) = O_{i^{k+1}}({t^{k+1}}, x^k, s^{k}_{{i^{k+1}}}, \xi^{k+1}), \quad \xi^{k+1} \sim \mathcal{D}_{{i^{k+1}}}$ \hfill $\rhd\,s^{k+1}_j = s^{k}_j \quad \forall j \neq i^{k+1}$
  \ENDFOR
  \end{algorithmic}
\end{protocol}
\begin{definition}[Function Class $\cF^{\textnormal{conv}}_{Q, L}$]\ \\
    \leavevmode
    We assume that a function $f \,:\,\R^d \rightarrow \R$ is convex, differentiable, $L$-smooth on the set $Q$, i.e.,
    \begin{align*}
        \norm{\nabla f(x) - \nabla f(y)} \leq L \norm{x - y} \quad \forall x, y \in Q.
    \end{align*}
    A set of all functions with such properties we define as $\cF^{\textnormal{conv}}_{Q, L}.$
    \label{def:func_class_convex}
\end{definition}

\begin{definition}[Algorithm Class $\cA_{\textnormal{zr}}$]\ \\
  \label{def:time_algorithm}
  An algorithm $A = \{A^k\}_{k=0}^{\infty}$ is a sequence such that
  \begin{align*}
      \squeeze
      \markchanges
      &A^k\,:\, \underbrace{\R^d \times \dots \times \R^d}_{k \textnormal{ times}} \rightarrow {\R_{\geq 0}} \times \R^d \quad \forall k \geq 1, A^0 \in {\R_{\geq 0}} \times \R^d,
  \end{align*}
  and, for all $k \geq 1$ and $g^1, \dots, g^{k} \in \R^d,$
  $
  t^{k+1} \geq t^{k},
  $
  where $t^{k+1}$ and $t^{k}$ are defined as $(t^{k+1}, \cdot) = A^{k}(g^1, \dots, g^{k})$ and $(t^{k}, \cdot) = A^{k-1}(g^1, \dots, g^{k-1}).$ 
  Moreover, $x^k \in Q$ for all $k \geq 0.$
\end{definition}

The following oracle helps to formalize the fixed computation model.

\begin{align*}
  &O_{\tau}^{{{\scriptscriptstyle \nabla}f}}\,:\, \underbrace{\R_{\geq 0}}_{\textnormal{time}} \times \underbrace{\R^d}_{\textnormal{point}} \times \underbrace{(\R_{\geq 0} \times \R^d \times \{0, 1\})}_{\textnormal{input state}} \times \mathbb{S}_{\xi} \rightarrow \underbrace{(\R_{\geq 0} \times \R^d \times \{0, 1\})}_{\textnormal{output state}} \times \R^d
\end{align*}

\begin{align}
\begin{split}
  \label{eq:oracle_stochastic_delay}
  \squeeze
  &\textnormal{such that } O_{\tau}^{{{\scriptscriptstyle \nabla}f}}(t, x, (s_t, s_x, s_q), \xi) = \left\{
\begin{aligned}
  &((t, x, 1), &0), \qquad & s_q = 0, \\
  &((s_t, s_x, 1), &0), \qquad & s_q = 1 \textnormal{ and } t < s_t + \tau,\\
  &((0, 0, 0), & \nabla f(s_x; \xi)), \qquad & s_q = 1 \textnormal{ and } t \geq s_t + \tau,
\end{aligned}
\right.
\end{split}
\end{align}
and $\nabla f(\cdot;\cdot)$ is a stochastic mapping.

\begin{definition}[Oracle Class $\cO_{\tau_1, \dots, \tau_n}^{\textnormal{conv}, \sigma^2}$]\ \\
    Let us consider an oracle class such that, for any $f \in \cF^{\textnormal{conv}}_{Q, L},$ it returns oracles $O_i = O_{\tau_i}^{{{\scriptscriptstyle \nabla}f_i}}$ and distributions $\mathcal{D}_i$ for all $i \in [n],$ where $\nabla f_i(\cdot;\cdot)$ is an unbiased $\sigma^2$-variance-bounded mapping on the set $Q$ of the gradient of the local function in worker $i$. The oracles $O_{\tau_i}^{{{\scriptscriptstyle \nabla}f_i}}$ are defined in \eqref{eq:oracle_stochastic_delay}. We define such oracle class as $\cO_{\tau_1, \dots, \tau_n}^{\textnormal{conv}, \sigma^2}.$ Without loss of generality, we assume that $0 < \tau_1 \leq \dots \leq \tau_n.$ \label{def:oracle_class_convex}
\end{definition}
Notice that this oracle class differs from the oracle class for convex functions in \citep{tyurin2023optimal} because we consider the heterogeneous setting where the oracles return unbiased stochastic gradients of the local functions $f_i,$ which can be different. We refer the reader to \citep{tyurin2023optimal} for additional details about the time complexities formalization. We are now ready to state the theorem.

\begin{theorem}[Informal theorem (see the formal Theorem~\ref{theorem:lower_bound_homog_convex})]
  \label{thm:first_lower_bound}
  Let Assumptions~\ref{ass:convex}, \ref{ass:global_lipschitz_constant}, and \ref{ass:stochastic_variance_bounded} hold. It is impossible to converge faster than 
  \begin{align*}
    \Theta\left(\tau_n \nicefrac{\sqrt{L} R}{\sqrt{\varepsilon}} + \left(\nicefrac{1}{n} \sum_{i=1}^{n} \tau_{i}\right) \nicefrac{\sigma^2 R^2}{n \varepsilon^2}\right)
   \end{align*}
  seconds under the fixed computation model.
\end{theorem}

\begin{restatable}{theorem}{THEOREMLOWERBOUNDCONVEX}
    \label{theorem:lower_bound_homog_convex}
    Let us consider the oracle class $\cO_{\tau_1, \dots, \tau_n}^{\textnormal{conv}, \sigma^2}$ for some $\sigma^2 >0$ and $0 < \tau_1 \leq \dots \leq \tau_n.$ We fix any $R, L, \varepsilon > 0$ such that $\sqrt{L} R > c_1 \sqrt{\varepsilon} > 0.$ For any 
   $$ t \leq c \times \left[\tau_n \frac{\sqrt{L} R}{\sqrt{\varepsilon}} + \left(\frac{1}{n} \sum\limits_{i=1}^{n} \tau_{i}\right) \frac{\sigma^2 R^2}{n \varepsilon^2}\right],$$ 
    in the view Protocol~\ref{alg:time_multiple_oracle_protocol}, for any algorithm $A \in \cA_{\textnormal{zr}},$ there exists a set $Q,$ a function $f \in \cF^{\textnormal{conv}}_{Q, L}$ and oracles and distributions $((O_1, \dots, O_n), (\mathcal{D}_1, \dots, \mathcal{D}_n)) \in \cO_{\tau_1, \dots, \tau_n}^{\textnormal{conv}, \sigma^2}(f)$
    such that
        $$\Exp{f(x^{k(t)})} - \inf_{x \in Q} f(x) > \varepsilon,$$
    where $k(t)$ is the largest index such that $t^{k(t)} \leq t,$
    and $R$ is the euclidean distance between $0$ (starting point) and the closest solution $x_* \in Q.$ 
    The quantities $c_1,$ and $c$ are {\markchanges universal} constants.
\end{restatable}

\begin{proof}
\emph{First term.} It is easy to prove the dependence on the first term 
$\tau_n \frac{\sqrt{L} R}{\sqrt{\varepsilon}}$ 
using the same idea as in \citep{lu2021optimal,tyurin2023optimal,huang2022lower}. It is sufficient to put a ``hard'' convex function \citep{nesterov2018lectures,woodworth2018graph} to the slowest worker corresponding with the time $\tau_n = \max_{i \in [n]} \tau_i.$ In particular, we can consider the ``hard'' quadratic function $\bar{f}$ from \citep{nesterov2018lectures}[Section 2.1.2] and take the functions
  \begin{align*}
    f_i(x) = \begin{cases}
      n \times \bar{f}(x), & i = n \\
      0, & i < n.
    \end{cases}
  \end{align*}
  for all $x \in \R^d.$ The function $f = \frac{1}{n} \sum_{i=1}^{n} f_i = \bar{f}$ belongs to the class $\cF^{\textnormal{conv}}_{Q, L}.$ We take the stochastic gradients without noise, i.e., $\nabla f_i(x;\xi_i) = \nabla f_i(x)$ deterministically for all $x \in \R^d,$ $\xi_i \in \mathbb{S}_{\xi_i},$ and $i \in [n].$ It is clear that the only worker that can solve the problem is worker $n,$ and it takes $\tau_n$ seconds to find one gradient by the oracle construction. Thus, the required time complexity is $\Theta\left(\tau_n \frac{\sqrt{L} R}{\sqrt{\varepsilon}}\right)$ since the required oracle complexity is $\Theta\left(\frac{\sqrt{L} R}{\sqrt{\varepsilon}}\right)$ \citep{nesterov2018lectures}. 

  \emph{Second term.} The proof of the second term is slightly trickier and uses the construction from \citep{woodworth2018graph}. Let us fix any algorithm. We use the proof of Lemma~10 from \citep{woodworth2018graph} that has the following result. For any $\sigma^2,$ $B > 0$ and any algorithm, it is possible to construct a \emph{one dimensional linear} function $g \,:\, \R \to \R$ on the domain $\{x \in \R \,:\, \abs{x} \leq B\},$ a stochastic gradient mapping $\nabla g \,:\, \R \times \mathbb{S}_{\xi} \to \R,$ and a distribution $\mathcal{D}$ such that 
  \begin{align}
   \label{eq:TDQiLLesGeHinEQfcj}
    \Exp{g(x^N)} - \min_{\abs{x} \leq B} g(x) \geq \frac{\sigma B}{8 \sqrt{N}}
  \end{align}
  after $N$ queries of the oracle, where $\nabla g$ is unbiased and $\sigma^2$-variance-bounded. 
  
  The idea is to put a function $g_i$ to each worker but with different domain sizes. In particular, for all $i \in [n],$ we take the function $f_i \,:\, \R^n \to \R$ such that
  \begin{align}
    \label{eq:JXuOOHssn}
    f_i(x) = g_i(x_i),
  \end{align}
  where $g_i$ is the function from Lemma~10 of \citep{woodworth2018graph} applied independently with $B=R_i$ and $N=N_i(\bar{t}),$ and $x_i$ is the $i$\textsuperscript{th} coordinate of a vector $x.$ For all $i \in [n],$ we consider the function $f_i$ on the domain $\{x_i \in \R \,|\, \abs{x_i} \leq R_i\},$ where $R_i \eqdef R \times \frac{\sqrt{\tau_i}}{\sqrt{\sum_{i=1}^n \tau_i}}.$ One can see that $f$ is convex, $0$--smooth (because $g_i$ is linear).
  The distance between $0$ and the optimal point is less or equal to $R$ because
  \begin{align*}
    \sum_{i=1}^n R_i^2 = \sum_{i=1}^n R^2 \frac{\tau_i}{\sum_{i=1}^n \tau_i} = R^2
  \end{align*}
  and the optimal point for the problem $g_i(x_i) \rightarrow \min_{\abs{x_i} \leq R_i}$ is either $R_i$ or $-R_i.$ We take
  $$Q = \{x \in \R^n \,:\, \abs{x_i} \leq R_i \quad \forall i \in [n]\}.$$

  Let us define the time 
  \begin{align}
    \label{eq:TDQiLLesGeHinEQfcj3}
    \bar{t} \eqdef \frac{\sigma^2 R^2}{256 n \varepsilon^2} \left(\frac{1}{n} \sum_{i=1}^{n} \tau_i\right).
  \end{align}
  By the time $\bar{t},$ worker $i$ can calculate at most
  \begin{align}
    \label{eq:TDQiLLesGeHinEQfcj2}
    N_i(\bar{t}) \eqdef \flr{\frac{\bar{t}}{\tau_i}}
  \end{align}
  stochastic gradients. Therefore,
  \begin{align*}
    \Exp{f(\bar{x})} - \min_{x \in Q} f(x) 
    &= \frac{1}{n} \sum_{i=1}^{n} \left(\Exp{g_i(\bar{x}_i)} - \min_{\abs{x_i} \leq R_i} g_i(x_i)\right) \overset{\eqref{eq:TDQiLLesGeHinEQfcj}}{\geq} \frac{1}{n} \sum_{i=1}^{n} \frac{\sigma R_i}{8 \sqrt{N_i(\bar{t})}} \\
    &= \frac{1}{n} \sum_{i=1}^{n} \frac{\sigma R \sqrt{\tau_i}}{8 \sqrt{N_i(\bar{t})} \sqrt{\sum_{i=1}^n \tau_i}} \overset{\eqref{eq:TDQiLLesGeHinEQfcj3}, \eqref{eq:TDQiLLesGeHinEQfcj2}}{\geq} \sum_{i=1}^{n} \frac{2 \varepsilon \tau_i}{\sum_{i=1}^n \tau_i} = 2 \varepsilon.
  \end{align*}
  where $\bar{x}$ is any possible output of the algorithm before the time $\bar{t}.$

\end{proof}

\section{Proof of Theorems~\ref{thm:rennala} and \ref{thm:malenia}}

\begin{restatable}{theorem}{RENNALATHEOREM}
  \label{thm:rennala}
  Let Assumptions~\ref{ass:global_lipschitz_constant}, \ref{ass:convex}, \ref{ass:stochastic_variance_bounded}, and \ref{ass:pl_global} hold, \textbf{and the functions $\{f_i\}$ are equal}. Let us take $\gamma = \nicefrac{1}{L}$ and $S = \nicefrac{4 \sigma^2}{\mu L \varepsilon},$ then \algname{Rennala SGD} (Algorithm~\ref{alg:alg_server_heterog} with $w_i^k = \nicefrac{n}{\sum_{i=1}^n B_i^k}$) finds $x^{k+1}$ such that $\Exp{\norm{x^{k+1} - x^{k+1}_*}^2} \leq \varepsilon$ after
  \begin{align}
    \label{eq:rennala_old}
    \squeeze \cO\left(\min\limits_{m \in [n]} \left[\left(\frac{1}{m} \sum\limits_{i=1}^{m} \frac{1}{\tau_{i}}\right)^{-1} \left(\frac{L}{\mu} + \frac{\sigma^2}{m \varepsilon \mu^2}\right)\right] \log \frac{R^2}{\varepsilon}\right)
  \end{align}
  seconds, where $x^{k+1}_*$ is the closest solution to $x^{k+1}.$
\end{restatable}
\begin{proof}
  Since the functions are equal, \algname{Rennala SGD} is equivalent to 
  \begin{equation*}
    \begin{aligned}
      & \squeeze x^{k+1} = x^{k} - \gamma g^k_{\textnormal{\algname{R}}}, \\ 
      & \squeeze g^k_{\textnormal{\algname{R}}} \eqdef \frac{1}{\sum_{i=1}^n B_i^k} \sum\limits_{i=1}^{n} \sum\limits_{j=1}^{B_i^k} \nabla f(x^k;\xi^k_{ij}).
    \end{aligned}
    \end{equation*}
    Clearly, $g^k_{\textnormal{\algname{R}}}$ is unbiased and 
    \begin{align*}
      \ExpSub{k}{\norm{g^k_{\textnormal{\algname{R}}} - \nabla f(x^k)}^2} = \left(\sum_{i=1}^n B_i^k\right)^{-2} \sum\limits_{i=1}^{n} \sum\limits_{j=1}^{B_i^k} \ExpSub{k}{\norm{\nabla f(x^k;\xi^k_{ij}) - \nabla f(x^k)}}^2 \leq \sigma^2 \left(\sum_{i=1}^n B_i^k\right)^{-1}.
    \end{align*}
    \algname{Rennala SGD} waits for the moment when $\sum_{i=1}^n B_i^k > S$ (see Alg.~\ref{alg:alg_server_heterog} with $w_i^k = \nicefrac{n}{\sum_{i=1}^n B_i^k}$). Thus
    \begin{align*}
      \ExpSub{k}{\norm{g^k_{\textnormal{\algname{R}}} - \nabla f(x^k)}^2} \leq \frac{\sigma^2}{S} \leq \frac{\mu L \varepsilon}{4}
    \end{align*}
    We can use Theorem~\ref{thm:aux} to get
    \begin{align*}
      \Exp{\norm{x^{k+1} - x^{k+1}_*}^2} \leq \left(1 - \frac{\gamma \mu}{2}\right)^{k+1} \norm{x^{0} - x^{k}_*}^2 + \frac{\gamma L \varepsilon}{2}.
    \end{align*}
    Since $\gamma = \frac{1}{L},$ we obtain
    \begin{align*}
      \Exp{\norm{x^{k+1} - x^{k+1}_*}^2} \leq \left(1 - \frac{\mu}{2 L}\right)^{k+1} \norm{x^{0} - x^{k}_*}^2 + \frac{\varepsilon}{2}.
    \end{align*}
    The last inequality ensure that the method finds an $\varepsilon$--solution after
    \begin{align*}
      \cO\left(\frac{L}{\mu} \log \frac{R^2}{\varepsilon}\right)
    \end{align*}
    iterations. In each iteration, the method has to ensure that $\sum_{i=1}^n B_i^k > S.$ A sufficient time for that is 
    \begin{align*}
      2 \min\limits_{m \in [n]} \left[\left(\frac{1}{m} \sum\limits_{i=1}^{m} \frac{1}{\tau_{i}}\right)^{-1} \left(1 + S\right)\right].
    \end{align*}
    under the fixed computation model (see Theorem~11 in \citep{tyurin2024freya}). It is left to multiply this time by $\cO\left(\frac{L}{\mu} \log \frac{R^2}{\varepsilon}\right).$
\end{proof}

\begin{restatable}{theorem}{MALENIATHEOREM}
  \label{thm:malenia}
  Let Assumptions~\ref{ass:global_lipschitz_constant}, \ref{ass:convex}, \ref{ass:stochastic_variance_bounded}, and \ref{ass:pl_global}  hold. Let us take $\gamma = \nicefrac{1}{L}$ and $S = \nicefrac{4 \sigma^2}{\mu L \varepsilon},$ then \algname{Malenia SGD} (Algorithm~\ref{alg:alg_server_heterog} with $w_i^k = \nicefrac{1}{B_i^k}$) finds $x^{k+1}$ such that $\Exp{\norm{x^{k+1} - x^{k+1}_*}^2} \leq \varepsilon$ after
  \begin{align}
    \label{eq:malenia_time}
    \squeeze \cO\left(\left[\tau_n \frac{L}{\mu} + \left(\frac{1}{n} \sum\limits_{i=1}^{n} \tau_{i}\right) \frac{\sigma^2}{n \varepsilon \mu^2} \right]\log \frac{R^2}{\varepsilon}\right)
  \end{align}
  seconds, where $x^{k+1}_*$ is the closest solution to $x^{k+1}.$
\end{restatable}

\begin{proof}
  The proof of this theorem almost repeats the proof of Theorem~\ref{thm:rennala}. The variance of \algname{Malenia SGD} is 
  \begin{align*}
    \ExpSub{k}{\norm{g^k_{\textnormal{\algname{M}}} - \nabla f(x^k)}^2} = \ExpSub{k}{\norm{\frac{1}{n} \sum\limits_{i=1}^{n} \frac{1}{B_i^k}  \sum\limits_{j=1}^{B_i^k} \nabla f_i(x^k;\xi^k_{ij}) - \nabla f(x^k)}^2} \leq \frac{\sigma^2}{n} \left(\frac{1}{n} \sum\limits_{j=1}^{n} \frac{1}{B_i^k}\right).
  \end{align*}
  The method waits for the moment when $\left(\frac{1}{n} \sum_{i=1}^n \nicefrac{1}{B_i^k}\right)^{-1} > \frac{S}{n}.$ Therefore
  \begin{align*}
    \ExpSub{k}{\norm{g^k_{\textnormal{\algname{M}}} - \nabla f(x^k)}^2} \leq \frac{\sigma^2}{S}.
  \end{align*}
  Using the same reasoning, the method finds an $\varepsilon$--solution after
  \begin{align*}
    \cO\left(\frac{L}{\mu} \log \frac{R^2}{\varepsilon}\right)
  \end{align*}
  iterations. In each iteration, the method has to ensure that $\left(\frac{1}{n} \sum_{i=1}^n \nicefrac{1}{B_i^k}\right)^{-1} > \frac{S}{n}.$ A sufficient time for that is
  \begin{align*}
    \bar{t} = 2 \left(\tau_n + \left(\frac{1}{n} \sum_{i=1}^n \tau_i \right) \frac{S}{n} \right)
  \end{align*}
  under the fixed computation model because the number of computed stochastic gradients $B_i^k \geq \flr{\frac{\bar{t}}{\tau_i}},$ and
  \begin{align*}
    \frac{1}{n} \sum_{i=1}^n \frac{1}{B_i^k} \leq \frac{1}{n} \sum_{i=1}^n \frac{1}{\flr{\frac{\bar{t}}{\tau_i}}} \leq \frac{1}{n} \sum_{i=1}^n \frac{2 \tau_i}{\bar{t}} < \frac{n}{S},
  \end{align*}
  where we use $\flr{x} \geq \frac{x}{2}$ for all $x \geq 1.$ Multiplying $\bar{t}$ by $\cO\left(\frac{L}{\mu} \log \frac{R^2}{\varepsilon}\right),$ we get the result.
\end{proof}

\begin{theorem}
  \label{thm:aux}
  Consider the method 
  \begin{equation}
    \begin{aligned}
      \squeeze x^{k+1} = x^{k} - \gamma \nabla f(x^k;\xi^k),
    \end{aligned}
    \label{eq:uYyDKbIYszbWwSqNpJB}
    \end{equation}
  where $\ExpSub{k}{\nabla f(x^k;\xi^k)} = \nabla f(x^k),$ $\ExpSub{k}{\norm{\nabla f(x^k;\xi^k) - \nabla f(x^k)}^2} \leq \sigma^2,$ and $\sigma^2 > 0.$
  Let Assumptions~\ref{ass:convex}, \ref{ass:pl_global}, and \ref{ass:global_lipschitz_constant} hold. Let us take $\gamma = \nicefrac{1}{L}$ and $S = \nicefrac{2 \sigma^2}{\mu L \varepsilon},$ then the method finds $x^{k+1}$ such that 
  \begin{align*}
    \Exp{\norm{x^{k+1} - x^{k+1}_*}^2} \leq \left(1 - \frac{\gamma \mu}{2}\right)^{k+1} \norm{x^{0} - x^{k}_*}^2 + \frac{2 \gamma \sigma^2}{\mu},
  \end{align*} where $x^{k+1}_*$ is the closest solution of $\min\limits_{x \in \R^d} f(x)$ to $x^{k+1}.$
\end{theorem}

\begin{proof}
  Using the properties of the projection and \eqref{eq:uYyDKbIYszbWwSqNpJB}, we have
  \begin{align*}
    \ExpSub{k}{\norm{x^{k+1} - x^{k+1}_*}^2} 
    &\leq \ExpSub{k}{\norm{x^{k+1} - x^{k}_*}^2} \\
    &= \ExpSub{k}{\norm{x^{k} - \gamma \nabla f(x^k;\xi^k) - x^{k}_*}^2} \\
    &= \ExpSub{k}{\norm{x^{k} - x^{k}_*}^2} - 2 \gamma \ExpSub{k}{\inp{\nabla f(x^k;\xi^k)}{x^{k} - x^{k}_*}} + \gamma^2 \ExpSub{k}{\norm{\nabla f(x^k;\xi^k)}^2} \\
    &= \ExpSub{k}{\norm{x^{k} - x^{k}_*}^2} - 2 \gamma \inp{\nabla f(x^k)}{x^{k} - x^{k}_*} + \gamma^2 \ExpSub{k}{\norm{\nabla f(x^k;\xi^k)}^2}.
  \end{align*}
  In the last equality, we use the unbiasedness. Due the variance decomposition equality, we get
  \begin{equation}
  \begin{aligned}
    \ExpSub{k}{\norm{x^{k+1} - x^{k+1}_*}^2} 
    &\leq \ExpSub{k}{\norm{x^{k} - x^{k}_*}^2} - 2 \gamma \inp{\nabla f(x^k)}{x^{k} - x^{k}_*} + \gamma^2 \norm{\nabla f(x^k)}^2 + \gamma^2 \ExpSub{k}{\norm{\nabla f(x^k;\xi^k) - \nabla f(x^k)}^2}.
  \end{aligned}
  \label{eq:glhRcJwVVZEhiznex}
  \end{equation}
  Since the function $f$ is $L$--smooth and $\nabla f(x^{k}_*) = 0,$ we obtain
  \begin{align*}
    &- 2 \gamma \inp{\nabla f(x^k)}{x^{k} - x^{k}_*} + \gamma^2 \norm{\nabla f(x^k)}^2 \\
    &=- 2 \gamma \inp{\nabla f(x^k) - \nabla f(x_*)}{x^{k} - x^{k}_*} + \gamma^2 \norm{\nabla f(x^k) - \nabla f(x_*)}^2 \\
    &\leq- 2 \gamma \inp{\nabla f(x^k) - \nabla f(x_*)}{x^{k} - x^{k}_*} + L \gamma^2 \inp{\nabla f(x^k) - \nabla f(x_*)}{x^{k} - x^{k}_*} \\
    &= \gamma \left(L \gamma - 2 \right)  \inp{\nabla f(x^k) - \nabla f(x_*)}{x^{k} - x^{k}_*}.
  \end{align*}
  Taking $\gamma \leq \frac{1}{L}$ and substituting the inequality to \eqref{eq:glhRcJwVVZEhiznex}, we get
  \begin{align*}
    \ExpSub{k}{\norm{x^{k+1} - x^{k+1}_*}^2} 
    &\leq \ExpSub{k}{\norm{x^{k} - x^{k}_*}^2} - \gamma \inp{\nabla f(x^k)}{x^{k} - x^{k}_*} + \gamma^2 \ExpSub{k}{\norm{\nabla f(x^k;\xi^k) - \nabla f(x^k)}^2}.
  \end{align*}
  The $\sigma^2$--variance bounded ensures that
  \begin{align*}
    \ExpSub{k}{\norm{x^{k+1} - x^{k+1}_*}^2} 
    &\leq \ExpSub{k}{\norm{x^{k} - x^{k}_*}^2} - \gamma \inp{\nabla f(x^k)}{x^{k} - x^{k}_*} + \gamma^2 \sigma^2.
  \end{align*}
  Due to convexity and Assumption~\ref{ass:pl_global}, we can use Lemma~\ref{lemma:star_strongly_convex}, which yields
  \begin{align*}
    \ExpSub{k}{\norm{x^{k+1} - x^{k+1}_*}^2} 
    &\leq \ExpSub{k}{\norm{x^{k} - x^{k}_*}^2} - \frac{\gamma \mu}{2} \norm{x^{k} - x^{k}_*} + \gamma^2 \sigma^2 \\
    &= \left(1 - \frac{\gamma \mu}{2}\right)\ExpSub{k}{\norm{x^{k} - x^{k}_*}^2} + \gamma^2 \sigma^2.
  \end{align*}
  Unrolling the recursion and taking the full expectation, we obtain
  \begin{align*}
    \Exp{\norm{x^{k+1} - x^{k+1}_*}^2} \leq \left(1 - \frac{\gamma \mu}{2}\right)^{k+1} \norm{x^{0} - x^{0}_*}^2 + \frac{2 \gamma \sigma^2}{\mu}
  \end{align*}
\end{proof}

\section{Assumptions~\ref{ass:convex}, \ref{ass:inter_strong} and \ref{ass:pl_condition} imply Assumption~\ref{ass:pl_global}}
\label{sec:aux_imply}
\begin{theorem}
  Let $\{f_i\}$ satisfy Assumption~\ref{ass:convex}, \ref{ass:inter_strong}, and Assumption~\ref{ass:pl_condition} with constant $\mu$, then $f$ satisfies Assumption~\ref{ass:pl_global} with constant $\frac{\mu}{4}$.
\end{theorem}
\begin{proof}
  We fix $x \in \R^d.$ 
  Since Assumption~\ref{ass:inter_strong} hold, 
  then the functions share the closest solution $x_*$ to $x.$  Assumption~\ref{ass:pl_condition} ensures that
  \begin{align*}
    f_i(x) - f_i(x_*) \geq \frac{\mu}{2} \norm{x - x_*}^2.
  \end{align*}
  for all $i \in [n]$ \citep{karimi2016linear}. Thus
  \begin{align*}
    f(x) - f(x_*) \geq \frac{\mu}{2} \norm{x - x_*}^2.
  \end{align*}
  Due to convexity, we get
  \begin{align*}
    f(x_*) \geq f(x) + \inp{\nabla f(x)}{x_* - x}.
  \end{align*}
  Therefore
  \begin{align*}
    f(x) - f(x_*) \leq \inp{\nabla f(x)}{x - x_*} \leq \norm{\nabla f(x)} \norm{x - x_*} \leq \norm{\nabla f(x)} \sqrt{\frac{2}{\mu}}\sqrt{f(x) - f(x_*)}
  \end{align*}
  and 
  \begin{align*}
    \frac{\mu}{4} \left(f(x) - f(x_*)\right) \leq \frac{1}{2} \norm{\nabla f(x)}^2,
  \end{align*}
  which is Assumption~\ref{ass:pl_global} with constant $\frac{\mu}{4}.$
\end{proof}

\newpage

\section{Experiments}
\label{sec:exp}
We conduct a comparison between \algname{Rennala SGD} and \algname{Malenia SGD} on both stochastic quadratic optimization tasks and real-world machine learning problems. These are standard quadratic optimization and computer vision problems, the design of which we explain in Section~\ref{sec:exp_details}. We developed a library that simulates the behavior of $n = 100$ workers. Both methods have two hyperparameters: step size $\gamma$ and parameter $S.$ We do a grid search for both methods and find the best pairs in all setups.
We start with synthetic quadratic optimization problems, which are generated \emph{without and with the interpolation regime}. The procedure is described in Section~\ref{sec:exp_more}.

\subsection{Without interpolation}
\begin{figure}[H]
  \centering
  \includegraphics[width=0.5\columnwidth]{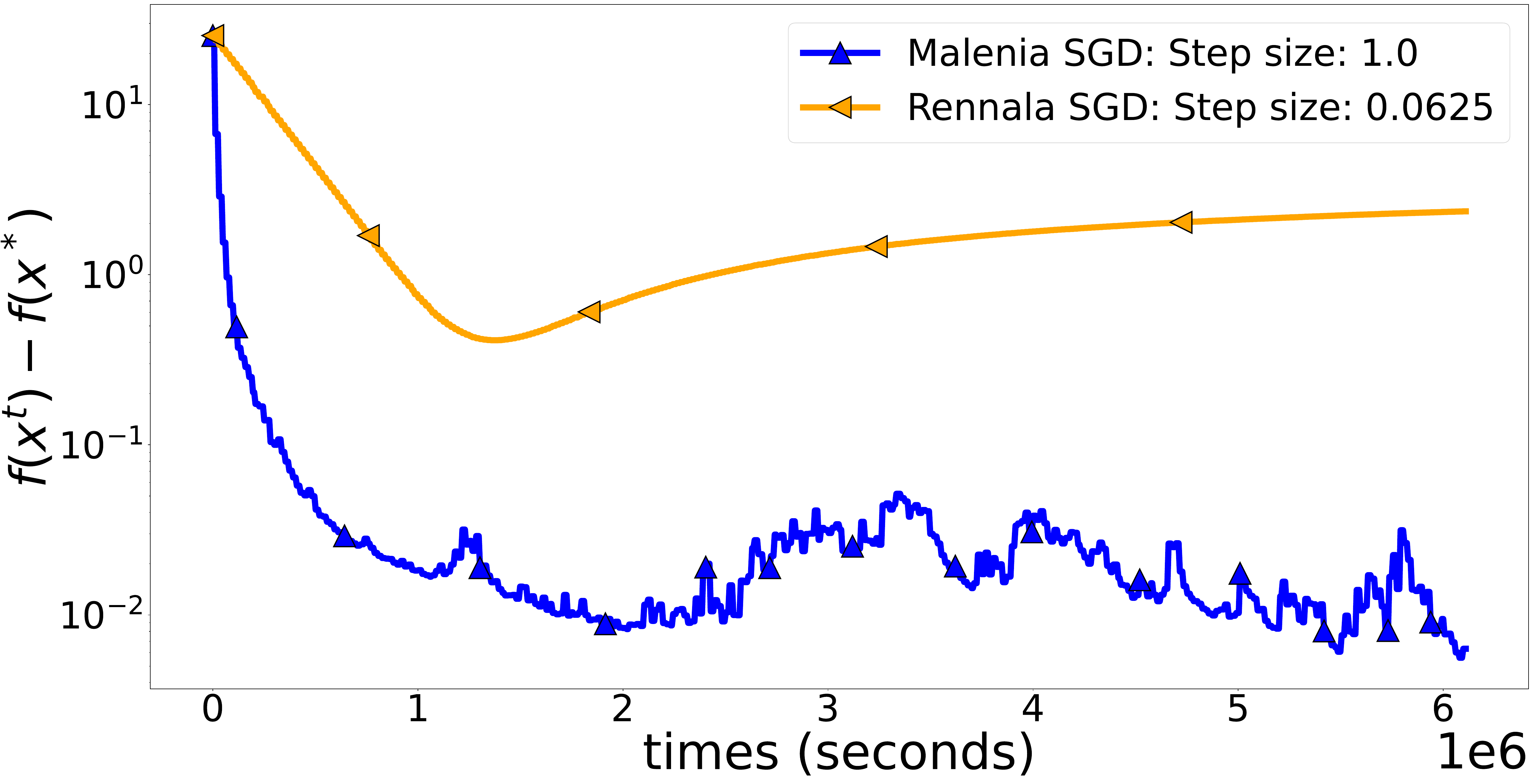}
  \caption{Comparison of the methods on a quadratic optimization problem \emph{without interpolation}. We take the computation time $\tau_i = i^2$ for all $i \in [n].$}
  \label{fig:no_inter}
\end{figure}
\label{sec:exp:without}
In Figure~\ref{fig:no_inter}, we present results without interpolation. The plots concur with the theory from Section~\ref{sec:challenges}, where we explain that it is essential to have interpolation to break the time complexity of \algname{Malenia SGD}. \algname{Rennala SGD} has biased gradient estimators and does not converge to a minimum of the quadratic optimization problem in Figure~\ref{fig:no_inter}.

\subsection{With interpolation}
\label{sec:exp:with}
\begin{figure}[H]
  \centering
  \begin{subfigure}[t]{0.48\columnwidth}
    \centering
    \includegraphics[width=\columnwidth]{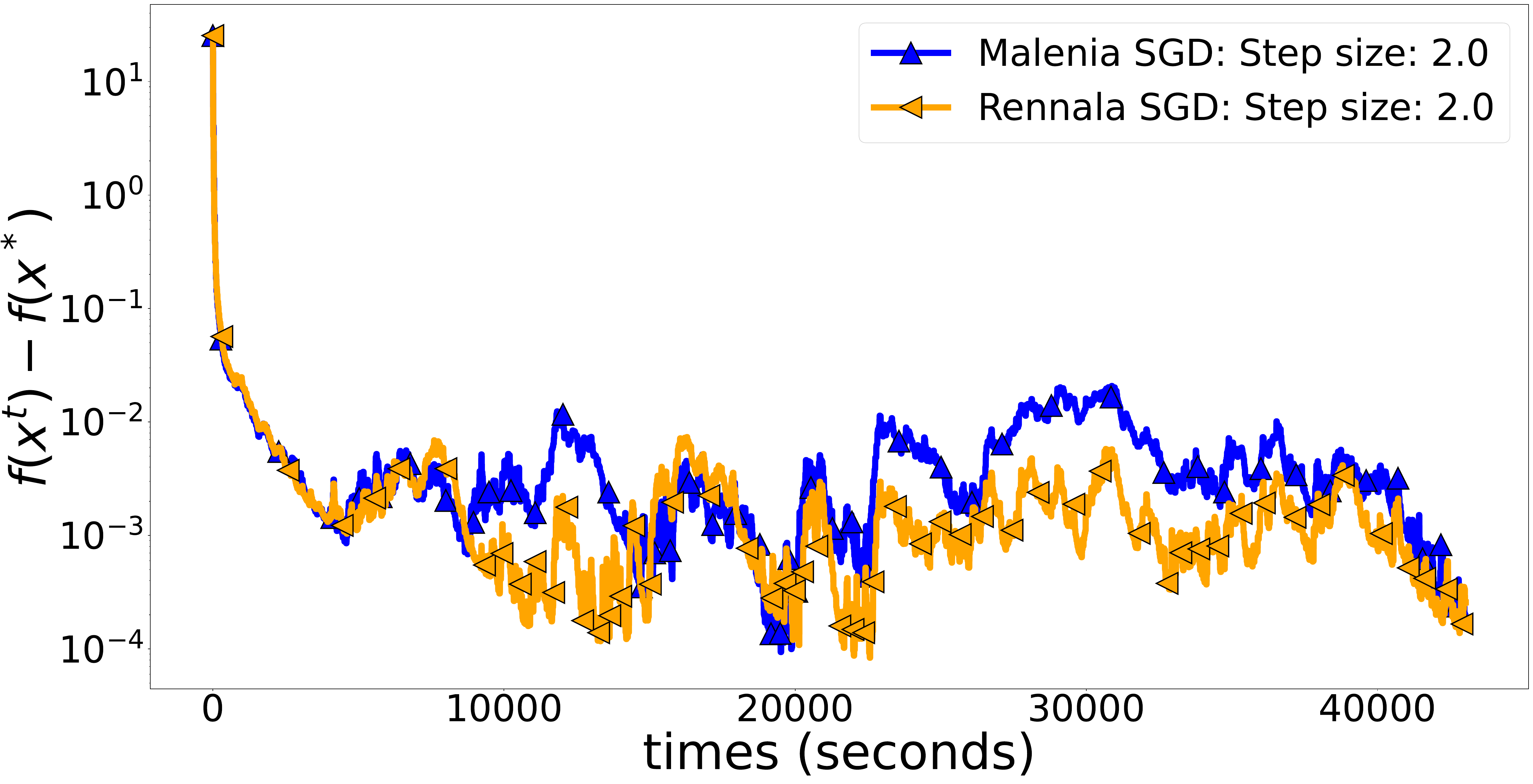}
  \end{subfigure}%
  \begin{subfigure}[t]{0.48\columnwidth}
    \centering
    \includegraphics[width=\columnwidth]{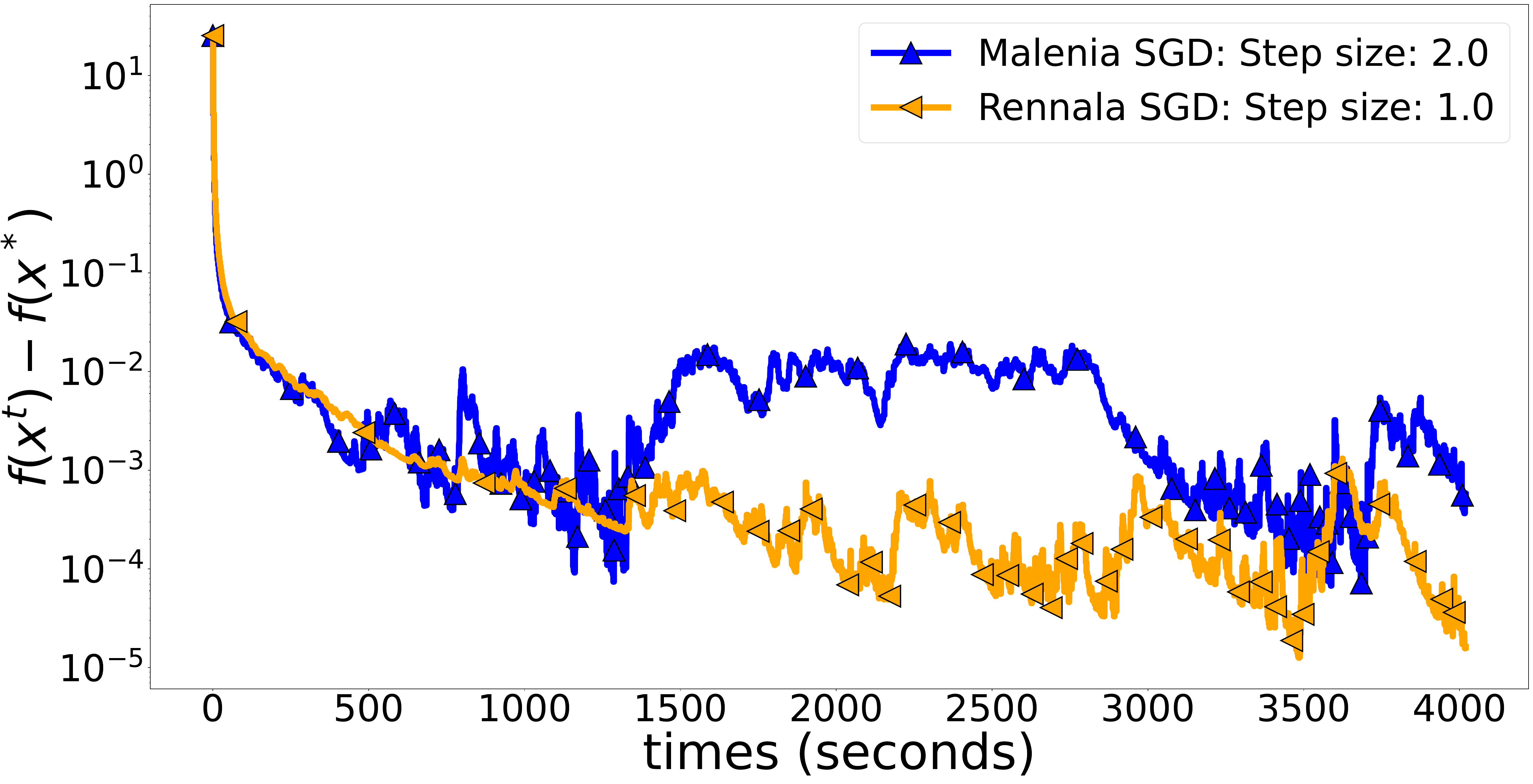}
  \end{subfigure}
  \caption{Comparison of the methods on quadratic optimization problems \emph{with interpolation}. Times $\{\tau_i\}$ less diverse: \emph{Left plot:} $\tau_i = \sqrt{i}$ for all $i \in [n].$ \emph{Right plot:} $\tau_1 = 0.01, \tau_2 = 1, \dots, \tau_n = 1.$}
  \label{fig:log}
\end{figure}

In Figures~\ref{fig:log} and \ref{fig:log_2}, we consider the methods in the interpolation regime. As expected, according to Section~\ref{sec:main}, \algname{Rennala SGD} outperforms \algname{Malenia SGD} in all experiments. We compare the methods with different $\{\tau_i\}.$ In Figures~\ref{fig:log}, the times $\{\tau_i\}$ are less diverse, so the difference between the methods is less profound. In Figures~\ref{fig:log_2}, $\{\tau_i\}$ are more different; thus, we can see that \algname{Rennala SGD} converges much faster to low function values because it has much less variance in the corresponding gradient estimator.

\begin{figure}[H]
  \centering
  \begin{subfigure}[t]{0.48\columnwidth}
    \centering
    \includegraphics[width=\columnwidth]{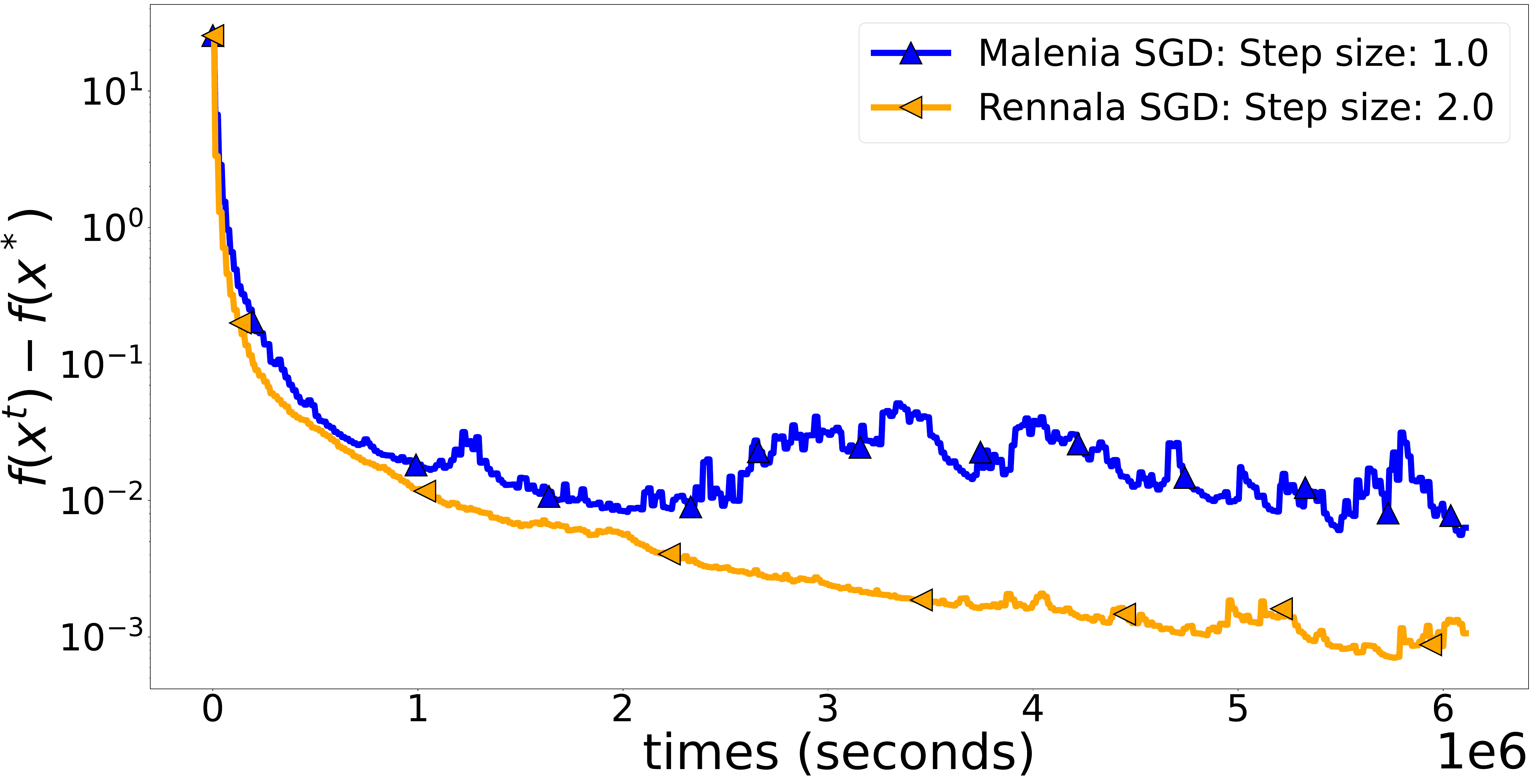}
  \end{subfigure}%
  \begin{subfigure}[t]{0.48\columnwidth}
    \centering
    \includegraphics[width=\columnwidth]{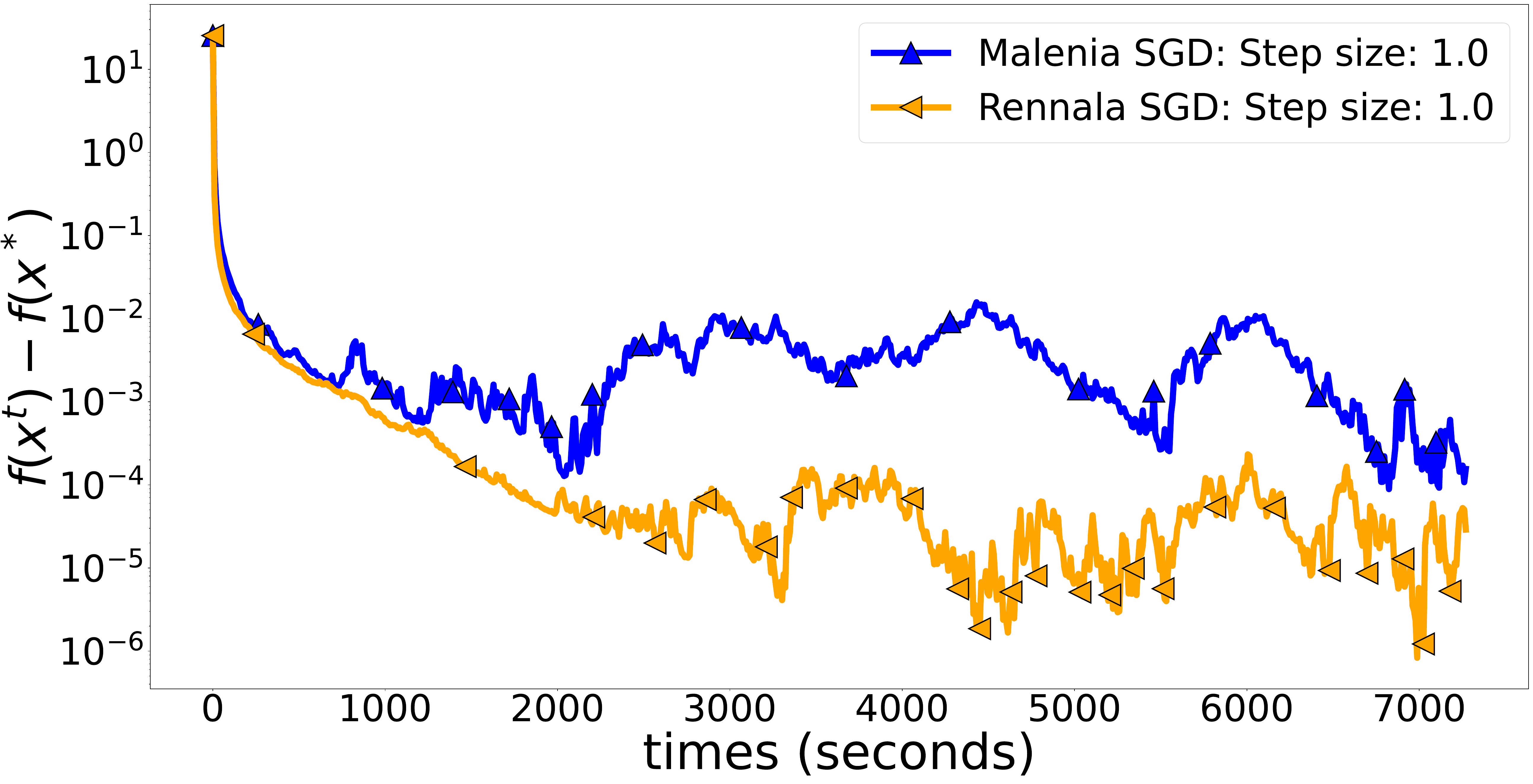}
  \end{subfigure}
  \caption{Comparison of the methods on quadratic optimization problems \emph{with interpolation}. Times $\{\tau_i\}$ more diverse: \emph{Left plot:} $\tau_i = i^2$ for all $i \in [n].$ \emph{Right plot:} $\tau_1 = 0.001, \tau_2 = 1, \dots, \tau_n = 1.$}
  \label{fig:log_2}
\end{figure}

\subsection{ResNet-18 and CIFAR-10}
\label{sec:resnet}
We also verify how \algname{Rennala SGD} and \algname{Malenia SGD} work with ResNet-18 and the CIFAR-10 classification problem \citep{krizhevsky2009learning} (License: MIT). Both algorithms take step size $\gamma = 0.25,$ sample a batch of size $128,$ and the smallest $S$ such that all workers calculate at least one batch. The dataset CIFAR-10 is split between the workers, so we consider the heterogeneous setting; all workers access different samples.
The results of the experiments are presented in Figure~\ref{fig:nn}. One can see that \algname{Rennala SGD} converges faster in terms of accuracy, which might be explained by the fact that neural networks work in the interpolation regime. Note that this is an empirical observation in the nonconvex setup, and explaining it from the theoretical point of view is an important future work.

\begin{figure}[H]
  \centering
  \includegraphics[width=0.5\columnwidth]{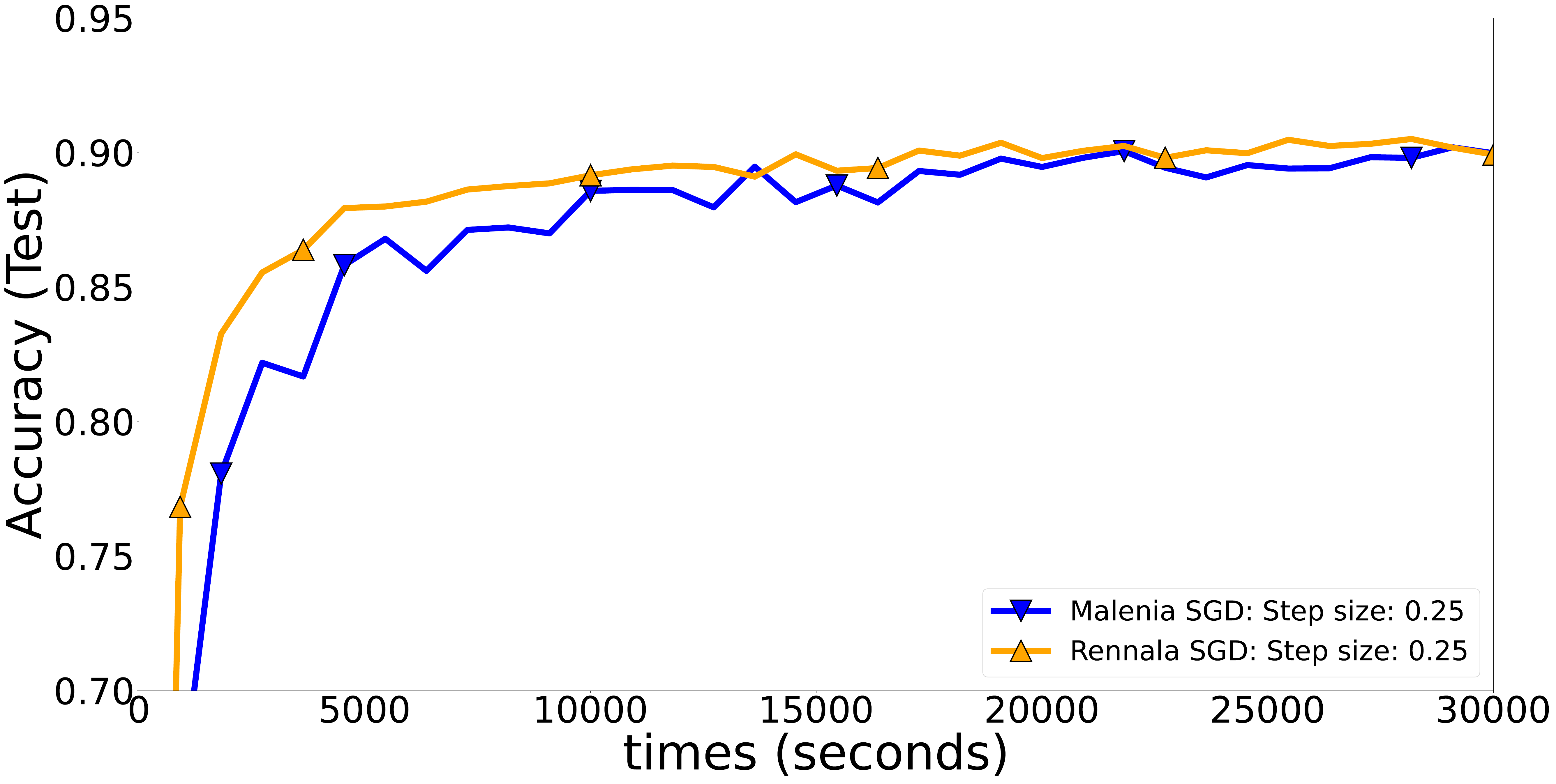}
  \caption{Comparison of the methods on the CIFAR-10 classification problem with ResNet-18. We take the computation time $\tau_i = i^2.$}
  \label{fig:nn}
\end{figure}

\newpage
\clearpage

\section{Experiments Details}
\label{sec:exp_details}
The experiments were run in Python 3 using an Intel(R) Xeon(R) Gold 6248 CPU @ 2.50GHz.

\subsection{Quadratic optimization task generation procedure}
\label{sec:exp_more}
In Section~\ref{sec:exp}, we perform experiments using synthetic quadratic optimization problems
\begin{align*}
  \min_{x \in \R^d} \frac{1}{n} \sum_{i=1}^n \left(\frac{1}{2}x^\top \mA_i x - x^\top b_i\right).
\end{align*}
Below, we present the algorithm, based on \citep{szlendak2021permutation}, that generates these problems. In all experiments, we take $s = 3$ to ensure that the generated matrices are diverse. We take $n = 100,$ $d = 100,$ and $\lambda = 0.001.$ The stochastic gradients are equal to the true gradients plus standard Gaussian noise added to the coordinates to emulate stochasticity.

With these parameters and procedures, we run the experiments from Section~\ref{sec:exp:without}. To conduct the experiments from Section~\ref{sec:exp:with} in the interpolation regime, we take the matrices $\mA_1, \cdots, \mA_n$, vectors $b_1, \cdots, b_n$ returned by Algorithm~\ref{alg:gen}. Let $\bar{x}_*$ be the solution of the quadratic optimization problem $\frac{1}{n} \sum_{i=1}^{n} \mA_i \bar{x}_* = \frac{1}{n} \sum_{i=1}^{n} b_i.$ Then, we redefine the vectors $\{b_i\}$ as $b_i = \mA_i \bar{x}_*$ to ensure that we are working in the interpolation regime. With this strategy, the matrices are still different, and the functions $\{f_i\}$ are not equal.

\begin{algorithm}[!h]
  \caption{Generate quadratic optimization tasks}
  \label{alg:gen}
  \begin{algorithmic}[1]
  \label{algorithm:matrix_generation}
  \STATE \textbf{Parameters:} number nodes $n$, dimension $d$, regularizer $\lambda$, and noise scale $s$.
  \FOR{$i = 1, \dots, n$}
  \STATE Generate random noises $\eta_i^s = 1 + s \zeta_i^s$ and $\eta_i^b = s \zeta_i^b,$ i.i.d. $\zeta_i^s, \zeta_i^b \sim \mathcal{N}(0, 1)$
  \STATE Take vector $b_i = \frac{\eta_i^s}{4}(-1 + \eta_i^b, 0, \cdots, 0) \in \R^{d}$
  \STATE Take the initial tridiagonal matrix
  \[\mA_i = \frac{\eta_i^s}{4}\left( \begin{array}{cccc}
    2 & -1 & & 0\\
    -1 & \ddots & \ddots & \\
    & \ddots & \ddots & -1 \\
    0 & & -1 & 2 \end{array} \right) \in \R^{d \times d}\]
  \ENDFOR
  \STATE Take the mean of matrices $\mA = \frac{1}{n}\sum_{i=1}^n \mA_i$
  \STATE Find the minimum eigenvalue $\lambda_{\min}(\mA)$
  \FOR{$i = 1, \dots, n$}
  \STATE Update matrix $\mA_i = \mA_i + (\lambda - \lambda_{\min}(\mA)) \mI$
  \ENDFOR
  \STATE Take starting point $x^0 = (\sqrt{d}, 0, \cdots, 0)$
  \STATE \textbf{Output:} matrices $\mA_1, \cdots, \mA_n$, vectors $b_1, \cdots, b_n$, starting point $x^0$
  \end{algorithmic}
\end{algorithm}

\subsection{Experiments with ResNet and CIFAR-10}

In Section~\ref{sec:resnet}, we consider the standard computer vision classification problem with ResNet-18 \citep{he2016deep} and CIFAR-10 \citep{krizhevsky2009learning}. We conduct the experiments using PyTorch and implement both \algname{Rennala SGD} and \algname{Malenia SGD} optimizers. For reproducibility, we use the default ResNet-18 architecture provided in PyTorch and split randomly and evenly the CIFAR-10 dataset across multiple workers to create a heterogeneous data distribution scenario. We use standard preprocessing techniques for CIFAR-10, including normalization and random cropping, and train the network for a fixed number of epochs. The performance metrics include top-1 accuracy. In total, we solve the optimization problem
\begin{align*}
  \textstyle\min\limits_{x \in \R^d} \frac{1}{n} \sum_{i=1}^{n} \left(\frac{1}{m} \sum_{j=1}^{m} \textnormal{loss(ResNet}(a_{ij};x), y_{ij})\right),
\end{align*}
where ``loss'' is the standard cross-entropy loss, $\{a_{ij}, y_{ij}\}$ are samples from CIFAR-10 splitted between the workers.

\end{document}